\documentclass{article}

\usepackage{amsmath,amsfonts,amsthm,amstext,amssymb,amscd,amsbsy}
\usepackage[latin1]{inputenc}

\usepackage{indentfirst}
\usepackage[dvips]{color}
\usepackage{color}
\usepackage{epsfig}
\usepackage{hyperref}

\usepackage{psfrag}
\usepackage{lscape}
\usepackage{pdflscape}

\newtheorem{theorem}{Theorem}[section]
\newtheorem{proposition}[theorem]{Proposition}
\newtheorem{lemma}[theorem]{Lemma}
\newtheorem{corollary}[theorem]{Corollary}

\newtheorem{remark}[theorem]{Remark}

\newcommand{\CC}{\mathbb C}
\newcommand{\RR}{\mathbb R}

\newcommand{\mf}{\mathfrak}

\newcommand{\ad}{\mbox{ad}}
\newcommand{\F}{\mathbb{F}_{\Theta}}
\newcommand{\Pt}{\Pi_{\Theta}}
\newcommand{\p}{\mf{p}}
\newcommand{\m}{\mf{m}}
\newcommand{\pit}{\Pi_{\mf{t}}}

\begin{document} 
\title{Isotropy  summands and Einstein Equation of Invariant Metrics on Classical Flag Manifolds}

\author{Luciana Aparecida Alves \thanks{Federal University of Uberlândia, luciana.postigo@gmail.com} and Neiton Pereira da Silva \thanks{Federal University of Uberlândia, neitonps@gmail.com}}
\maketitle




\begin{abstract}
 
It is well known that the Einstein equation on a Riemannian flag manifold $(G/K,g)$ reduces to a algebraic system, if $g$ is a $G$-invariant metric. In this paper we described this system for all flag manifolds of a classical Lie group. We also determined the number of isotropy summands for all of these spaces and proved certain properties of the set of t-roots for flag manifolds of type $B_n$, $C_n$ and $D_n$.
\end{abstract}

  



\section*{Introduction}

A Riemannian manifold $(M,g)$ is called Einstein manifold if its Ricci tensor $Ric(g)$ satisfies the Einstein equation $Ric(g)=cg$. for some $c\in\RR$. The study of Einstein manifold is connected with several areas of mathematics and has important applications on physics, see \cite{Besse}. 

Let $G$ be a connected compact semi-simple Lie group and  $G/K$ a partial flag manifold, where $K$ is the centralize of a torus in $G$. It is well known that the Einstein equation of a $G$-invariant (or invariant) metric $g$ on a flag manifold $G/K$ reduces to a algebraic system. It is also known that $G/K$ admits a invariant K\"ahler Einstein metric associated to the canonical complex structure. The problem of determining invariant Einstein metrics non K\"ahler has been studied by several author, see \cite{Arv art}, \cite{art 1} and \cite{kimura}.

\cite{Bohm-Wang-Ziller} conjectured that if $G/H$ is a compact homogeneous space whose isotropy representation consists of pairwise inequivalent irreducible summands, then the algebraic Einstein equations have only finitely many real solutions. These conjecture applies to flag manifolds and it is opened yet.

The number of unknowns and equations in the algebraic Einstein system is so great as  the number of isotropy summands of $G/K$. Thus \cite{Sakane}, \cite{Da Silva} and \cite{Arv-Chry} determined all invariant Einstein manifold for certain families of flag manifold with few isotropy summands. Another important problem is to determine whether two non-K\"ahler homogeneous Einstein metrics on a certain flag manifold are isometric or not. Recently \cite{Arv-Chry-Sakane} treated this problem on $Sp(n)/U(p)\times U(n-p)$, $1\leq p\leq n-1$, by computational means, obtaining the Gr\"obner basis associated to certain algebraic equations. 

In this paper we described explicitly the Einstein equation for a invariant metric $g$ on every flag manifold of a classical Lie group, see Propositions [\ref{Einstein equation Bn 1}, \ref{Einstein equation Bn 2},\ref{Einstein equation Cn 1}, \ref{Einstein equation Cn 2} and \ref{Einstein equation Dn 2}], which extend partially the papers \cite{Arv art} and \cite{Sakane}. We also determine the number of isotropy summands of all of these flag manifold, which provides an overview of the size of the algebraic Einstein system for such spaces, see Theorem \ref{number of isotropy summands}. It is known that the amount of isotropy summands can be determined computing a certain set of linear functional, called t-roots. We showed that the set of t-roots of flag manifolds of Lie groups of type $C_n$ and $D_n$ has a special property:
\begin{theorem}
The set $\Pi_{\mathfrak{t}}$ of t-roots corresponding to the flag manifolds $Sp(n)/ U\left(n_{1}\right) \times\cdots\times U\left( n_{s}\right)$ is a system of roots of type $C_{s}$.
\end{theorem}

\begin{theorem} \label{t-roots properties Dn}
The set of t-roots $\Pi_{\mathfrak{t}}$ corresponding to the flag manifolds
$SO(2n)/U(n_{1})\times\cdots\times U(n_{s})$ is a root system of type
$C_{s}$.
\end{theorem}

We obtained are results proceeding in a similar way of \cite{Arv art} and making several laborious computations by means of careful inspections.

\section{Preliminaries}

In this section we set up our notation and present the standard
theory of partial flag manifolds associated with semisimple
Lie algebras (see, for example, \cite{SM N}, \cite{nir e sofia}).

 Let $\mf{g}$ be a finite-dimensional semisimple complex Lie
algebra and take a Lie group $G$ with Lie algebra $\mf{g}$. Let
$\mf{h}$ be a Cartan subalgebra. We denote by $\Pi$ the system of
roots of $(\mf{g},\mf{h})$. A root $\alpha\in\Pi$ is a linear
functional on $\mf{g}$. It uniquely determines an element
$H_{\alpha}\in\mf{h}$ by the Riesz representation $\alpha(X)=B(
X,H)$, $X\in\mf{g}$, with respect to the Cartan-Killing form
$B(\cdot,\cdot)$. The Lie algebra $\mf{g}$ has the following
decomposition
\[
\mf{g}=\mf{h}\oplus\sum_{\alpha\in\Pi}\mf{g}_{\alpha}
\]
where $\mf{g}_{\alpha}$ is the one-dimensional root space corresponding to $\alpha$.

We fix a system $\Sigma$ of simple roots of $\Pi$ and denote by
$\Pi^+$ and $\Pi^{-}$ the corresponding set of positive and
negative roots, respectively. Let $\Theta\subset\Sigma$ be a
subset, define
\begin{align}
\Pi_{\Theta}:=\langle\Theta\rangle\cap\Pi\hspace{1.5cm}
\Pt^\pm:=\langle\Theta\rangle\cap\Pi^{\pm}.\nonumber
\end{align}
We denote by $\Pi_M:=\Pi\setminus\Pt$ the complementary set of roots. Note that
\[
\p_{\Theta}:=\mf{h}\oplus\sum_{\alpha\in\Pi^+}\mf{g}_{\alpha}\oplus\sum_{\alpha\in\Pi_{\Theta}^-}\mf{g}_{\alpha}
\]
is a parabolic subalgebra, since it contains the Borel subalgebra
$\mf{b}^+=\mf{h}\oplus\sum\limits_{\alpha\in\Pi^+}\mf{g}_{\alpha}$.

The partial flag manifold determined by the choice $\Theta\subset \Pi$ is the
homogeneous space $\F=G/P_{\Theta}$, where $P_{\Theta}$ is the normalizer of
$\mf{p}_{\Theta}$ in $G$. In the special case $\Theta=\emptyset$, we obtain the \emph{full} (or maximal) flag manifold $\mathbb{F}=G/B$ associated with $\Pi$, where
$B$ is the normalize of the Borel subalgebra $\mf{b}^+=\mf{h}\oplus\sum\limits_{\alpha\in\Pi^+}\mf{g}_{\alpha}$ in $G$. For further use, for each $\alpha\in\Pi_M$, define the following sets

\begin{equation}\label{sets}
\Pt(\alpha):=\left\{\phi\in \Pt:(\alpha+\phi)\in\Pi\right\}\quad \text{and}\quad
\Pi_M(\alpha):=\{\beta\in \Pi_M:(\alpha+\beta)\in\Pi_M\}.
\end{equation}

Now we will discuss the construction of any flag manifold as the
quotient $U/K_{\Theta}$ of a semisimple compact Lie group
$U\subset G$ modulo the centralizer $K_{\Theta}$ of a torus in
$U$. We fix once and for all a Weyl base of $\mf{g}$ which amounts
to giving $X_\alpha\in\mf{g_{\alpha}}$, $H_{\alpha}\in\mf{h}$ with $\alpha\in\Pi$, with the standard
properties:
\begin{equation}\label{base weyl}
\begin{tabular}
[c]{lll}
$B( X_{\alpha},X_{\beta}) =\left\{
\begin{array}
[c]{cc}%
1, & \alpha+\beta=0,\\
0, & \text{otherwise};
\end{array}
\right. $ & & $\left[  X_{\alpha},X_{\beta}\right]  =\left\{
\begin{array}
[c]{cc}%
H_{\alpha}\in\mathfrak{h}, & \alpha+\beta=0,\\
N_{\alpha,\beta}X_{\alpha+\beta}, & \alpha+\beta\in\Pi,\\
0, & \text{otherwise.}%
\end{array}
\right.  $
\end{tabular}
\end{equation}
The real numbers $N_{\alpha,\beta}$ are non-zero if and only if
$\alpha+\beta\in\Pi$. Besides that, it satisfies
$$
\left\{
\begin{array}
[c]{cc}%
 N_{\alpha,\beta}=-N_{-\alpha,-\beta}=-N_{\beta,\alpha},& \\
\hspace*{-1cm} N_{\alpha,\beta}=N_{\beta,\gamma}=N_{\gamma,\alpha},&\mbox{se }\alpha+\beta+\gamma=0.
\end{array}
\right.
$$


We consider the following two-dimensional real spaces
$\mf{u}_{\alpha}=\text{span}_\mathbb{R}\{A_{\alpha},S_{\alpha}\}$
where $\hspace{0.3cm}A_{\alpha}=X_{\alpha}-X_{-\alpha}$
and $\hspace{0.3cm}S_{\alpha}=i(X_{\alpha}+X_{-\alpha})$, with
$\alpha\in\Pi^{+}$.
Then the real Lie algebra
$\mf{u}=i\mf{h}_\mathbb{R}\oplus\sum\mf{u}_{\alpha}$, with $\alpha\in\Pi^{+},$
is a compact real form of $\mf{g}$, where $\mf{h}_\mathbb{R}$ denotes the subspace of $\mf{h}$ spaned by $\{H_\alpha, \alpha\in \Pi\}$.

Let $U=\exp\mf{u}$ be the compact real form of $G$ corresponding
to $\mf{u}$. By the restriction of the action of $G$ on $\F$, we
can see that $U$ acts transitively on $\F$ then $\F=U/K_{\Theta}$,
where $K_{\Theta}=P_{\Theta}\cap U$. The Lie algebra of $K_\Theta$ is the set of fixed points of the conjugation $\tau\colon X_{\alpha}\mapsto-X_{-\alpha}$ of $\mf{g}$ restricted to $\p_{\Theta}$
$$
\mf{k}_{\Theta}=\mf{u}\cap\p_{\Theta}=i\mf{h}_\mathbb{R}\oplus\sum_{\alpha\in\Pt^+}\mf{u}_{\alpha}.
$$

The tangent space of $\F=U/K_{\Theta}$ at the origin $o=eK_{\Theta}$ can be identified with the orthogonal complement (with respect to the Killing form) of $\mf{k}_{\Theta}$ in $\mf{u}$
$$
T_o\F=\m=\sum\limits_{\alpha\in\Pi_{M}^{+}}\mf{u}_{\alpha}
$$
where $\Pi_{M}^{+}=\Pi_{M}\cap\Pi^+$. Thus we have $\mf{u}=\mf{k}_{\Theta}\oplus\mf{m}$.

On the other hand, there exists a nice way to decompose the tangent space $\m$. It is a standard factor that $\F$ is a reductive homogeneous space, this means that the adjoint representation of $\mf{k}_{\Theta}$ and $K_{\Theta}$ leaves $\mf{m}$ invariant, i.e. $\ad(\mf{k}_{\Theta})\mf{m}\subset\mf{m}$. Thus we can decompose $\mf{m}$ into a sum of irreducible $\ad (\mf{k}_{\Theta})$ submodules $\mf{m}_i$ of the module $\mf{m}$
\[
\m= \m_1\oplus\cdots\oplus\m_s.
\]
Each irreducible $\ad (\mf{k}_{\Theta})$ submodules can be described by means of certain linear functional called t-roots (see \cite{Alek e Perol} or \cite{Sie}).

By complexifying the Lie algebra of $K_{\Theta}$ we obtain
$$
\mf{k}_{\Theta}^{\mathbb{C}}=\mf{h}\oplus\sum_{\alpha\in\Pt}\mf{g}_{\alpha}.
$$
The adjoint representation of $\ad(\mf{k}_{\Theta}^{\CC})$ of $\mf{k}_{\Theta}^{\CC}$ leaves the complex tangent space $\m^{\CC}$ invariant.
Let
$$
\mf{t}:=Z(\mf{k}_{\Theta}^\CC)\cap i\mf{h}_\RR%
$$
be the intersection of the center of the parabolic subalgebra
$\mf{k}_{\Theta}^\CC$ with $i\mf{h}_\RR$. According to \cite{Arv art}, we can write
$$
\mf{t}=\{H\in i\mf{h}_{\RR}:\alpha(H)=0,\,\text{for all} \,\alpha \in \Pt\}.
$$

Let $i\mf{h}_\RR^{\ast}$ and $\mf{t}^{\ast}$ be the dual vector
space of $i\mf{h}_\RR$ and $\mf{t}$, respectively, and consider the
map $k\colon i\mf{h}_\RR^{\ast}\longrightarrow\mf{t}^{\ast}$ given
by $k(\alpha)=\alpha|_{\mf{t}}$. The linear functionals of
$\Pi_{\mf{t}}:=k(\Pi_{M})$ are called \emph{t-roots}. We denote by $\Pi_{\mf{t}}^+=k(\Pi_M^+)$ the set of positive t-roots. There exists a 1-1 correspondence between positive t-roots and irreducible submodules of the adjoint representation of $\mf{k}_{\Theta}$. This correspondence is given by
\begin{equation*}
\xi\longleftrightarrow\mf{m}_{\xi}=\sum_{k(\alpha)=\xi}\mf{u}_{\alpha}
\end{equation*}
with $\xi\in\Pi_{\mf{t}}^+$ (see \cite{Alek e Perol}). Hence the tangent space can be decomposed as following
\begin{equation*}\label{decomp}
\m=\m_{\xi_1}\oplus\cdots\oplus\m_{\xi_s}
\end{equation*}
where $\pit^+=\lbrace \xi_1, \ldots, \xi_s\rbrace$.

It is well known that a Riemannian invariant metric on $\F$ is completely determined by a real inner product $g\left(\cdot,\cdot\right)$ on $\mathfrak{m}=T_{o}\F$ which is invariant by the adjoint action of $\mf{k}_{\Theta}$. Besides that, any real inner product $\ad(\mf{k}_{\Theta})$-invariant on $\mf{m}$ has the form

\begin{equation}\label{inner product}
g\left(\cdot,\cdot\right)=-\lambda_1 B\left(\cdot,\cdot\right)|_{\mathfrak{m}_{1}\times\mathfrak{m}_{1}}-\cdots
-\lambda_s B\left(\cdot,\cdot\right)|_{\mathfrak{m}_{s}\times\mathfrak{m}_{s}}
\end{equation}\\
where ${\mathfrak{m}}_{i}=\mf{m}_{\xi_i}$ and $\lambda_i=\lambda_{\xi_i}>0$ with $\xi_{i}\in\Pi^{+}_{\mathfrak{t}}$, for $i=1,\ldots,s$. So any invariant Riemannian metric on $\F$ is determined by $|\Pi_{\mf{t}}^+|$ positive parameters. We will call an inner product defined by (\ref{inner product}) as an invariant metric on $\F$.

In a similar way, the Ricci tensor $Ric_g(\cdot.\cdot)$ of a invariant metric on $\F$ depends on $|\Pi_{\mf{t}}^+|$ parameters. Actually, it has the form
\[
Ric_g\left(\cdot,\cdot\right)=-r_ 1\lambda_1 B\left(\cdot,\cdot\right)|_{\mathfrak{m}_{1}\times\mathfrak{m}_{1}}-\cdots
-r_s\lambda_s B\left(\cdot,\cdot\right)|_{\mathfrak{m}_{s}\times\mathfrak{m}_{s}}.
\]
Then an invariant metric $g$ on $\F$ is Einstein iff $r_1=\cdots=r_s$. The next result shows a way to compute the components of the Ricci tensor by means of certain vectors of Weyl base.

\begin{proposition}(\cite{Arv art})
\label{Ricci} The Ricci tensor for an invariant metric $g$  on $\F$ is
given by
\begin{align}
Ric\left(  X_{\alpha},X_{\beta}\right)
=0,\quad\alpha,\beta\in\Pi
_{M},\alpha+\beta\notin\Pi_{M},\nonumber\\
& \nonumber\\
Ric(X_{\alpha},X_{-\alpha}) =B\left(  \alpha,\alpha\right) +\sum
_{\substack{\phi\in\Pt\\\alpha+\phi\in\Pi}}N_{\alpha,\phi}^{2}+\frac{1}%
{4}\sum_{\substack{\beta\in\Pi_{M}\\\alpha+\beta\in\Pi_{M}}}\frac{N_{\alpha,\beta}^{2}}{\lambda_{\alpha+\beta
}\lambda_{\beta}}\left( \lambda_{\alpha}^{2}-\left(
\lambda_{\alpha+\beta}-\lambda_{\beta}\right) ^{2}\right).
\label{Ric}
\end{align}
\end{proposition}
Since $Ric(\kappa
g)=Ric(g)$ ($\kappa\in\mathbb{R}$) we can normalize the Einstein equation $Ric(g)=c\cdot g$
choosing an appropriate value for $c$ or for some $\lambda_{\alpha}$.
\begin{remark}\label{remark}
Although (\ref{Ric}) is not in terms of t-roots, it is known that if $\alpha, \beta\in\Pi_M$ are two different roots that determine the same t-root, i.e.  $k(\alpha)=k(\beta)$, then $\lambda_{\alpha}=\lambda_{\beta}$ and $Ric(X_{\alpha},X_{-\alpha})=Ric(X_{\beta},X_{-\beta})$.
\end{remark}

\section{Einstein Equation}
\subsection{Case $A_{n}$}

Flag manifolds of type $A_{n}$ are homogeneous spaces of the special \\ unitary group, it has the form

\[
\mathbb{F}_{A}(n_{1},\ldots,n_{s})=SU(n)/S\left( U\left(
n_{1}\right) \times \cdots\times U\left( n_{s}\right)
\right), \quad  n=\sum\limits_{i=1}^{s}n_{i}.
\]
When each $n_i=1$ we obtain the full flag manifold $\mathbb{F}_A(n)=SU(n)/S\left( U\left(
1\right) \times \cdots\times U\left( 1\right)
\right)$.
\begin{proposition} (\cite{Arv art}) The Einstein equations on
$\mathbb{F}_{A}(n_{1},\ldots,n_{s})$ reduce to the following
algebraic system
\begin{equation}
n_{i}+n_{j}+\frac{1}{2}\sum_{l \neq
i,j}\frac{n_{l}}{\lambda_{il}\lambda_{jl}}(\lambda_{ij}^{2}-(\lambda_{il}-\lambda_{jl})^{2})=\lambda_{ij}\hspace{1cm}1\leq
i<j\leq s\label{equacoes Al}
\end{equation}
where the number of equations and unknowns $\lambda_{ij}$ is
$s(s-1)/2$.
\end{proposition}

\subsection{Case $B_{n}$}

There exist two types of flag manifolds of Lie groups $B_n$, which we denote by

\begin{align*}
&\mathbb{F}_{B}(n_{1},\ldots,n_{s})=SO\left( 2n+1\right)/U\left( n_{1}\right) \times\cdots\times U\left( n_{s}\right),
\quad n\geq2\quad\text{and}\quad s\leq n;\\ \\
&\mathbb{F}_{B}[n_{1},\ldots,n_{s+1}]=SO(2n+1)/U(n_{1})\times\cdots\times
U(n_{s})\times SO(2n_{s+1}+1),
\quad n_{s+1}\geq 2;
\end{align*}
where $n=\sum\limits_{i=1}^s n_{i}$.

First, we consider the case
 $\mathbb{F}_{B}(n_{1},\ldots,n_{s})$. The algebra $\mathfrak{so}\left(
2n+1,\mathbb{R}\right)$ is a compact real form of $\mf{g}=\mathfrak{so}\left(
2n+1,\mathbb{C}\right)$. The matrices of the Lie algebra $\mathfrak{so}\left(
2n+1,\mathbb{C}\right)$ have the following form
\[
A=
\begin{pmatrix}
0 & \beta & \gamma\\
-\gamma^{t} & a & b\\
-\beta^{t} & c & -a^{t}%
\end{pmatrix}
\]
where $\beta$ and $\gamma$ are matrices $1\times n$ and, $b$ and $c$ are skew symmetric matrices $n\times n$.
A Cartan subalgebra of $\mf{g}$ is given by
\begin{equation}\label{Cartan subalgebra Bn}
\mathfrak{h}=\{ \mathrm{diag}\left(  0,\varepsilon_{1}%
,\ldots,\varepsilon_{n},-\varepsilon_{1},\ldots,-\varepsilon_{n}\right)
; \,\varepsilon_{i}\in\mathbb{C}\}.
\end{equation}
The root system of the pair $(\mf{g},\mathfrak{h})$ is given by
\begin{equation}\label{Pi Bn}
\Pi =\{  \pm\varepsilon_{i};\,1\leq i\leq n\}
\cup\{ \pm\left(  \varepsilon_{i}-\varepsilon_{j}\right)
;\,1\leq i<j\leq n\} \cup\{  \pm\left(
\varepsilon_{i}+\varepsilon_{j}\right);\,1\leq i\neq j\leq
n\}
\end{equation}
where $\varepsilon_{i}$ denotes the functional $\mathrm{diag}\left(  0,\varepsilon_{1}
,\ldots,\varepsilon_{n},-\varepsilon_{1},\ldots,-\varepsilon_{n}\right)\mapsto\varepsilon_{i}$
and $\varepsilon_{i}\pm\varepsilon_{j}$ denotes the functional $\mathrm{diag}\left(  0,\varepsilon_{1}%
,\ldots,\varepsilon_{n},-\varepsilon_{1},\ldots,-\varepsilon_{n}\right)\mapsto\varepsilon_{i}\pm\varepsilon_{j}$.
A choice of positive roots is
\begin{equation}\label{positive root Bn}
\Pi^{+}=\{  \varepsilon_{i};\,1\leq i\leq n\}
\cup\{ \varepsilon_{i}-\varepsilon_{j};\,1\leq i<j\leq
n\} \cup\{ \varepsilon_{i}+\varepsilon_{j};\, 1\leq
i\neq j\leq n\}.
\end{equation}
The root system of the pair $\left(
\mathfrak{k}_{\Theta}^{\mathbb{C}}=  \mathfrak{sl}\left(
n_{1},\mathbb{C}\right)\times\cdots\times\mathfrak{sl} \left(
n_{s},\mathbb{C}\right),\mathfrak{h}\right)$ is given by
$$
\Pt =\{  \pm\left(  \varepsilon_{a}^{i}-\varepsilon_{b}
^{i}\right)  :1\leq a<b\leq n_{i}\}, 
$$
where $\varepsilon_{a}^{i}=\varepsilon_{n_{1}+\cdots+n_{i-1}+a}$. Then

\begin{align*}
\Pi_{M}^{+}  &  =\{  \varepsilon_{a}^{i};\,1\leq i\leq s,1\leq
a\leq n_{i}\} \cup\{
\varepsilon_{a}^{i}-\varepsilon_{b}^{j};\,1\leq i<j\leq s\}\\
\cup &  \{  \varepsilon_{a}^{i}+\varepsilon_{b}^{j};\,1\leq
i\neq j\leq s\} \cup\{
\varepsilon_{a}^{i}+\varepsilon_{b}^{i};\,1\leq a\neq b\leq
n_{i}\}  .
\end{align*}
The subalgebra $\mf{t}$
has the form
\begin{align}
\mathfrak{t} =\left\{
\begin{pmatrix}
0 &  & \\
& \Lambda & \\
&  & -\Lambda
\end{pmatrix}\in i\mathfrak{h}_{\RR}\label{t em Bl}
\right\},
\end{align}
where $
\Lambda 
=\mathrm{diag}\left(\varepsilon^{1}_{n_{1}},\ldots,\varepsilon^{1}_{n_{1}},
\varepsilon^{2}_{n_{2}},\ldots,\varepsilon^{2}_{n_{2}},\ldots,\varepsilon^{s}_{n_{s}}
,\ldots,\varepsilon^{s}_{n_{s}}\right)$ and $\varepsilon^{i}_{n_{i}}$ appears exactly
$n_{i}$ times, $i=1,\ldots,s$.

We will denote the t-root
$k(\varepsilon_{a}^{i})$ by $\delta_{i}$;
$k(\varepsilon_{a}^{i}\pm\varepsilon_{b}^{j})$ by
$\delta_{i}\pm\delta_{j}$; and,
$k(\varepsilon_{a}^{i}+\varepsilon_{b}^{i})$ by $2\delta_{i}$.
Thus
\begin{equation}\label{t roots Bn}
\Pi^{+}_{\mathfrak{t}}=k\left(\Pi_{M}^{+}\right)=\{\delta_{i},2\delta_{i}:1\leq
i\leq s\}\cup\{\delta_{i}\pm\delta_{j};1\leq i<j\leq
s\}.
\end{equation}
Note that in this case $\Pi_{\mathfrak{t}}$ is not a root system, because $\pm\delta_{i}$ are not the only multiples of $\delta_{i}$ in $\Pi_{\mathfrak{t}}$.

For the case $\mathbb{F}_{B}[n_{1},\ldots,n_{s+1}]$,
we take the same Cartan subalgebra of  $\mathfrak{so}\left(
2n+1,\mathbb{C}\right)$ in (\ref{Cartan subalgebra Bn}). The root system
of the pair  $\left(  \mathfrak{k}_{\Theta}^{\mathbb{C}},\mathfrak{h}\right)$, with $\mathfrak{k}_{\Theta}^{\mathbb{C}}= \mathfrak{sl}\left(
n_{1},\mathbb{C}\right) \times\cdots\times\mathfrak{sl}\left(
n_{s},\mathbb{C}\right) \times\mathfrak{so}\left(
2n_{s+1}+1,\mathbb{C}\right)$, is given by
\begin{align*}
\Pt    =&\{  \pm\left(  \varepsilon_{a}^{i}-\varepsilon_{b}%
^{i}\right); \,1\leq i\leq s,1\leq a<b\leq n_{i}\}
\cup\{ \pm\varepsilon_{a}^{s+1};\,1\leq a\leq
n_{s+1}\} 
& \\
& \cup  \{  \pm\left(
\varepsilon_{a}^{s+1}-\varepsilon_{b}^{s+1}\right)
;\,1\leq a<b\leq n_{s+1}\}  \cup\{  \pm\left(  \varepsilon_{a}%
^{s+1}+\varepsilon_{b}^{s+1}\right)  ;\,1\leq a<b\leq
n_{s+1}\}  .
\end{align*}
Thus we obtain
\begin{align*}
\Pi_{M}=&  \{  \pm\varepsilon_{a}^{i};\,1\leq i\leq s,\quad 1\leq
a\leq
n_{i}\}\cup\{  \pm(  \varepsilon_{a}^{i}\pm\varepsilon_{b}
^{j})  ;\,1\leq i<j\leq s+1\},\\
& \cup \{  \pm\left(
\varepsilon_{a}^{i}+\varepsilon_{b}^{i}\right)  ;\,1\leq i\leq
s,\quad 1\leq a<b\leq n_{i}\}  .
\end{align*}

The subalgebra $\mf{t}$ is formed by diagonal matrices of
the form (\ref{t em Bl}), with
\begin{equation}
\Lambda
=\mathrm{diag}\left(\varepsilon^{1}_{n_{1}},\ldots,\varepsilon^{1}_{n_{1}},\varepsilon^{2}_{n_{2}},\ldots,\varepsilon^{2}_{n_{2}},\ldots,
\varepsilon^{s}_{n_{s}},\ldots,\varepsilon^{s}_{n_{s}},0,\ldots,0\right)\label{t
em Bl II}
\end{equation}
and as the case before each $\varepsilon^{i}_{n_{i}}$ appears exactly $n_{i}$ times,
$i=1,\ldots,s$. Restricting the roots of $\Pi_{M}^{+}$ to
$\mathfrak{t}$ and keeping the notation
$\delta_{i}:=k(\varepsilon^{i}_{a})$ we obtain $\Pi^{+}_{\mathfrak{t}}$ as in (\ref{t roots Bn}).
Thus we proved the following result.

\begin{lemma} \label{t-roots properties of Bn}
The set of t-roots on both spaces
$\mathbb{F}_{B}(n_{1},\ldots,n_{s})$ and
$\mathbb{F}_{B}[n_{1},\ldots,n_{s+1}]$ is given by
\[
\Pi_{\mathfrak{t}}=\{\pm\delta_{i},\pm2\delta_{i},1\leq i\leq
s\}\cup\{\pm(\delta_{i}\pm\delta_{j}),1\leq i<j\leq
s\}
\]
where $\delta_{i}$ represents the restriction of the functional
$\varepsilon_{a}^{i}$ to $\mathfrak{t}$. In particular, there exist $s^{2}+s$ positive t-roots in both cases.
\end{lemma}

Now we  can compute the algebraic Einstein system for the
manifold  $\mathbb{F}_{B}[n_{1},\ldots,n_{s+1}]$. The
Killing form on $\mathfrak{so}(2n+1,\mathbb{C})$ is given by
\[
B(X,Y)=2(2n-1)\sum_{i=1}^{n}a_{i}b_{i}=\mathrm{tr}(XY)(2n-1),
\]
for $X=\mathrm{diag}(0,a_{1},\ldots,a_{n},-a_{1},\ldots,-a_{n})$ and
$Y=\mathrm{diag}(0,b_{1} ,\ldots,b_{n},-b_{1},\ldots,-b_{n})$. \\ Using the
Killing form, we obtain
\[
H_{\alpha}=
\begin{pmatrix}
0 &  & \\
& \Lambda_{\alpha} & \\
&  & -\Lambda_{\alpha}%
\end{pmatrix}
\]
where $\Lambda_{\alpha}$ is a diagonal matrix  $n\times n$ and
\begin{align*}
\Lambda_{\varepsilon_{i}-\varepsilon_{j}}  &  =\frac{1}{2(2n-1)}
\mathrm{diag}(0,\ldots,1_{i},\ldots,-1_{j},\ldots,0);
& \\
\Lambda_{\varepsilon_{i}+\varepsilon_{j}}  &  =\frac{1}{2(2n-1)}
\mathrm{diag}(0,\ldots,1_{i},\ldots,1_{j},\ldots,0);
& \\
\Lambda_{\varepsilon_{i}}  &
=\frac{1}{2(2n-1)}\mathrm{diag}(0,\ldots,1_{i},\ldots,0).
\end{align*}
Then
\begin{align*}
B( \varepsilon_{i}-\varepsilon_{j},\varepsilon_{i}-\varepsilon
_{j})   &  =B(
\varepsilon_{i}+\varepsilon_{j},\varepsilon
_{i}+\varepsilon_{j})  =\frac{1}{2n-1}\text{ and } B(
\varepsilon_{i},\varepsilon_{i}) =\frac{1}{2\left(
2n-1\right)  }.
\end{align*}
Using the equation
\begin{equation}
\left[  E_{\alpha},E_{-\alpha}\right]  =B\left(
E_{\alpha},E_{-\alpha}\right) H_{\alpha}, \label{igualdade}
\end{equation}
we obtain the vectors satisfying (\ref{base weyl}), where $E_{\alpha}$ is the canonical eigenvector in
$\mathfrak{g}_{\alpha}$.
Thus

$$
\begin{array}{ccc}
E^{ij}:=X_{\varepsilon_{i}-\varepsilon_{j}}=\frac{1}{\sqrt{2(2n-1)}
}E_{\varepsilon_{i}-\varepsilon_{j}}, & E^{ji}:=X_{\varepsilon_{j}-\varepsilon_{i}}=-\frac{1}{\sqrt{2(2n-1)}
}E_{\varepsilon_{j}-\varepsilon_{i}}, & 1\leq  i\,<j\leq
n; \\  & & \\
G^{i}:=X_{\varepsilon_{i}}=\frac{1}{\sqrt{2(2n-1)}}E_{\varepsilon_{i}
}, & G^{-i}:=X_{-\varepsilon_{i}}=-\frac{1}{\sqrt{2(2n-1)}}E_{-\varepsilon_{i}}, & 1\leq i\leq n;
\\& & \\
F^{ij}:=X_{\varepsilon_{i}+\varepsilon_{j}}=\frac{1}{\sqrt{2(2n-1)}}E_{\varepsilon_{i}+\varepsilon_{j}},& F^{-ij}:=X_{-\left(
\varepsilon_{i}+\varepsilon_{j}\right) }=-\frac
{1}{\sqrt{2(2n-1)}}E_{-\left(
\varepsilon_{i}+\varepsilon_{j}\right) }, & 1\leq i\,\neq j\leq
n.
\end{array} 
$$ 
For the using of Proposition \ref{Ricci} it will be convenient to
adopt the following notation:
$$
\begin{array}
[c]{ccc} G_{a}^{i}=X_{\varepsilon_{a}^{i}},&
G_{-a}^{-i}=X_{-\varepsilon
_{a}^{i}},& 1\leq i\leq s,\\
E_{ab}^{ij} =X_{\varepsilon_{a}^{i}-\varepsilon_{b}^{j}},& &1\leq
i\neq j\leq s,\\
F_{ab}^{ij}  =X_{\varepsilon_{a}^{i}+\varepsilon_{b}^{j}},&
F_{-ab}^{-ij}=X_{-\left(
\varepsilon_{a}^{i}+\varepsilon_{b}^{j}\right)
},&1\leq i<j\leq s,\\
F_{ab}^{i}  =X_{\varepsilon_{a}^{i}+\varepsilon_{b}^{i}},&
F_{-ab}^{-i}=X_{-\left(
\varepsilon_{a}^{i}+\varepsilon_{b}^{i}\right) },&1\leq a\neq
b\leq n_{i}.
\end{array}
$$
We will denote the invariant scalar product $g$ valued
on the base $\{ X_{\alpha ,}\alpha\in\Pi_{M}\}$, by:
\begin{align*}
g_{ij}  &  =g\left(  E_{ab}^{ij},E_{ba}^{ji}\right)  ,\text{
}f_{ij}=g\left( F_{ab}^{ij},F_{-ab}^{-ij}\right)  ,\text{ }1\leq
i<j\leq
s,\label{notacao em Bl}\\
h_{i}  &  =g\left(  G_{a}^{i},G_{-a}^{-i}\right)  ,\text{
}l_{i}=g\left( F_{ab}^{i},F_{-ab}^{-i}\right)  ,\text{ }1\leq
i\leq s.\nonumber
\end{align*}
From (\ref{base weyl}) we conclude that on $\mathfrak{so}(2n+1,\mathbb{C})$, the non zero
square of the structure constants $N_{\alpha,\beta}$ are
\begin{equation}
N^{2}_{\alpha,\beta}=\frac {1}{2\left( 2n-1\right)}.\label{const
est Bl}
\end{equation}

\begin{remark}
The restriction of the roots
$\varepsilon_{a}^{i}-\varepsilon_{b}^{s+1}$, $\varepsilon_{a}
^{i}+\varepsilon_{b}^{s+1}$ and $\varepsilon_{a}^{i}$ to the
subalgebra $\mathfrak{t}$, in (\ref{t em Bl II}), are equal for
$1\leq i\leq s$. Then it follows from Remark \ref{remark} that
\begin{equation}
g_{i(s+1)}=f_{i(s+1)}=h_{i},\hspace{0.5cm}1\leq i\leq s
\label{fi(s+1)=gi(s+1)}.
\end{equation}
\end{remark}

\newpage
\begin{proposition}
\label{equacoes Bl} \label{Einstein equation Bn 1} The Einstein equations on flag
manifolds $\mathbb{F}_{B}[n_{1},\ldots,n_{s+1}]$  reduces to an algebraic system of $s^2+s$ unknowns and $s^2+s$ equations, given by
\begin{align*}
& 2(n_{k}+n_{t})+\left(
1+8n_{s+1}\right)  \frac{\left( g_{kt}^{2}-\left(
h_{t}-h_{k}\right) ^{2}\right)}{h_{k} h_{t}}+\sum_{i\neq k,t}^{s}\frac{n_{i}}{g_{ik} g_{it}}\left(
g_{kt}^{2}-\left( g_{ik}-g_{it}\right) ^{2}\right)\\ \\
&  +\sum_{i\neq
k,t}^{s}\frac{n_{i}}{f_{ik}f_{it} }\left( g_{kt}^{2}-\left(
f_{ik}-f_{it}\right)  ^{2}\right)+\frac{4\left(  n_{t}-1\right) }{f_{kt}l_{t}}\left(
g_{kt}^{2}-\left(  f_{kt}-l_{t}\right)  ^{2}\right) \\ \\
&+\frac{4\left(  n_{k}-1\right)  }{f_{kt}l_{k}}\left(
g_{kt}^{2}-\left( f_{kt}-l_{k}\right)  ^{2}\right)=g_{kt},\quad \footnotesize{1\leq k<t\leq s}
\end{align*}

\begin{align*}
&2(2n_{s+1}+n_{k})+\frac{\left(  n_{k}-1\right)  }{h_{k}l_{k}}\left(  h_{k}%
^{2}-\left(  h_{k}-l_{k}\right)  ^{2}\right)+\sum_{i\neq k}^{s}\frac{n_{i}}{h_{i}%
f_{ik}}\left(  h_{k}^{2}-\left(  h_{i}-f_{ik}\right)^{2}\right)\\ \\
&+\sum_{i\neq k}^{s}\frac{n_{i}}{h_{i}%
g_{ik}}\left(  h_{k}^{2}-\left(  h_{i}-g_{ik}\right)  ^{2}\right)=h_{k},\quad \footnotesize{1\leq k\leq s}
\end{align*}

\begin{align*}
&2(n_{k}+n_{t})+\left(
1+2n_{s+1}\right)   \frac{\left( f_{kt}^{2}-\left(
h_{k}-h_{t}\right) ^{2}\right)  }{h_{k} h_{t}}+\frac{\left(  n_{k}-1\right) }{l_{k}g_{kt} }\left(
f_{kt}^{2}-\left(  l_{k}-g_{kt}\right) ^{2}\right)\\ \\
&+\frac{\left(
n_{t}-1\right) }{l_{t}g_{kt}}\left( f_{kt}^{2}-\left(
l_{t}-g_{kt}\right) ^{2}\right)+\sum_{i\neq k,t}^{s}\frac{n_{i}}{f_{ik}
g_{it}}\left(  f_{kt}^{2}-\left(  f_{ik}-g_{it}\right)
^{2}\right)\\ \\
&+\sum_{i\neq k,t}^{s}\frac{n_{i}}{f_{it}g_{ik}%
}\left(  f_{kt}^{2}-\left(  f_{it}-g_{ik}\right)  ^{2}\right)=f_{kt},\quad \footnotesize{1\leq k<t\leq s}
\end{align*}

\begin{align*}
&2(n_{k}+2n_{s+1})
+\frac{\left( n_{k}-1\right) }{h_{k}l_{k}}\left( h_{k}^{2}-\left(
h_{k}-l_{k}\right) ^{2}\right)+\sum_{i\neq k}^{s}\frac{n_{i}}{h_{i}g_{ik}%
}\left(  h_{k}^{2}-\left(  h_{i}-g_{ik}\right)  ^{2}\right)\\ \\
&+\sum_{i\neq k}^{s}\frac{n_{i}}{h_{i}f_{ik}%
}\left(  h_{k}^{2}-\left(  h_{i}-f_{ik}\right)  ^{2}\right)=h_{k},\quad \footnotesize{1\leq k\leq s}
\end{align*}

\begin{align*}
4(n_{k}-1)+\frac{\left(
1+2n_{s+1}\right)l_{k}^{2}}{h_{k}^{2}}+2\sum_{i\neq
k}^{s}\frac{n_{i}}{g_{ik}f_{ik}}\left( l_{k}^{2}-\left(
g_{ik}-f_{ik}\right)  ^{2}\right)=l_{k},\quad \footnotesize{1\leq k\leq s}
\end{align*}

\begin{align*}
& 2(n_{k}+2n_{s+1})+\frac{\left( n_{k}-1\right) }{l_{k}h_{k}}\left(
h_{k}^{2}-\left( l_{k}-h_{k}\right)  ^{2}\right)+\sum_{i\neq k}^{s}\frac{n_{i}}{f_{ik}h_{i}}\left(
h_{k}^{2}-\left( f_{ik}-h_{i}\right)  ^{2}\right)\\ \\
& +\sum_{i\neq k}^{s}\frac{n_{i}}{g_{ik}h_{i}
}\left(  h_{k}^{2}-\left(  g_{ik}-h_{i}\right)  ^{2}\right)=h_{k},\quad \footnotesize{1\leq k\leq s}
\end{align*}
\end{proposition}

\begin{proof}
The strategy of the proof is simple: 1) compute
$\Pi_{K}\left(\alpha\right)$ and $\Pi_{\Theta}\left(\alpha\right)$ for
each $\alpha\in \Pi_{M}$; 2) to substitute the sum of the elements
of these sets and the structural constants (\ref{const est Bl}) in
the Proposition \ref{Ricci}, obtaining explicitly the Einstein
algebraic system. The laborious part is the step 1), thus we write only this step.

\begin{landscape}
\begin{small}
\begin{tabular}{ccccccc}
  \hline
  $\alpha\in \Pi_{M}$ &  &  &$\Pi_{\Theta}\left(\alpha\right)$ is the union of &  & &$\Pi_{M}\left(\alpha\right)$ is the union of  \\
  \hline
   $\varepsilon_{c}^{k}-\varepsilon_{d}^{t} $&  &  &$\{
\varepsilon_{a}^{k}-\varepsilon_{c}^{k}:\scriptsize{1\leq a\leq n_{k},a\neq
c}\}$   &  &  & $ \{  \left(  \varepsilon_{d}^{t}+\varepsilon_{a}^{k}\right)
,-\left( \varepsilon_{c}^{k}+\varepsilon_{a}^{k}\right):\scriptsize{1\leq
a\leq n_{k},a\neq c }\}, \{  \varepsilon_{d}^{t}\}$,
\\
\scriptsize{$\quad 1\leq k<t\leq s$}   &  &  &$\{
\varepsilon_{d}^{t}-\varepsilon_{a}^{t}:\scriptsize{ 1\leq a\leq n_{t},a\neq
d }\}$  &  &  & $\{  -\varepsilon_{c}^{k}\}, \{  \left(  \varepsilon_{d}^{t}+\varepsilon_{a}^{t}\right),-\left(
\varepsilon_{c}^{k}+\varepsilon_{a}^{t}\right):\scriptsize{1\leq a\leq
n_{t},a\neq d}\}$  \\
    &  &  &  &  &  & $ \{  \left(
\varepsilon_{d}^{t}\pm\varepsilon_{a}^{i}\right)  ,\left(
\varepsilon_{a}^{i}-\varepsilon_{c}^{k}\right)  ,-\left(  \varepsilon_{c}^{k}+\varepsilon_{a}^{i}\right)\}$  \\
&  &  &  &  &  & \footnotesize{$1\leq i\leq s+1,i\neq k,t \text{ and } 1\leq a\leq n_{i}$}  \\
\hline
$\varepsilon_{c}^{k}
-\varepsilon_{d}^{s+1}$&  &  &$\{  \left(
\varepsilon_{d}^{s+1}-\varepsilon_{a}^{s+1}\right) ,(
\varepsilon_{a}^{s+1}+\varepsilon_{d}^{s+1})\}$  &  &  & $\{  \left(  \varepsilon_{a}^{i}-\varepsilon_{c}^{k}\right),-\left(
\varepsilon_{a}^{i}+\varepsilon_{c}^{k}\right)  ,\left(  \varepsilon_{d}%
^{s+1}-\varepsilon_{a}^{i}\right)  ,\left(
\varepsilon_{a}^{i}+\varepsilon _{d}^{s+1}\right)\}$   \\\scriptsize{$1\leq k\leq s$}
&  &  &\footnotesize{ $1\leq a\leq n_{s+1},a\neq d$}   &  &  &\footnotesize{  $1\leq i\leq s,i\neq k,1\leq a\leq n_{i}$}   \\
&  &  & $\{
\varepsilon_{a}^{k}-\varepsilon_{c}^{k}:\scriptsize{1\leq a\leq n_{k},a\neq
c}\}$ &  &  & $ \{  -\left( \varepsilon_{a}^{k}+\varepsilon_{c}^{k}\right),\left( \varepsilon_{a}^{k}+\varepsilon_{d}^{s+1}\right):\scriptsize{1\leq
a\leq n_{k},a\neq c}\}$   \\
&  &  & $\{  \varepsilon_{d}^{s+1}\}$ &  &  &   \\
\hline
$\varepsilon_{c}
^{k}+\varepsilon_{d}^{t}$ &  &  & $\{  \varepsilon_{a}^{t}-\varepsilon_{d}^{t}:\scriptsize{1\leq a\leq n_{t},a\neq d}\}$ &  &  & $\{  \left(  \varepsilon_{a}^{k}-\varepsilon_{d}^{t}\right)  ,-\left(
\varepsilon_{c}^{k}+\varepsilon_{a}^{k}\right):\scriptsize{1\leq a\leq
n_{k},a\neq c}\},$  \\
\scriptsize{$1\leq k<t\leq s$} &  &  & $\{
\varepsilon_{a}^{k}-\varepsilon_{c}^{k}:\scriptsize{1\leq a\leq n_{k},a\neq
c}\} $ &  &  & $\{  -\varepsilon_{d}^{t}\}  ,\{  -\varepsilon_{c}
^{k}\} $  \\
&  &  &  &  &  & $ \{  \left(  \varepsilon_{a}^{t}-\varepsilon_{c}^{k}\right)
,-\left( \varepsilon_{d}^{t}+\varepsilon_{a}^{t}\right):\scriptsize{1\leq
a\leq n_{t},a\neq d}\},$   \\
&  &  &  &  &  & $ \{  \left(
\varepsilon_{a}^{s+1}-\varepsilon_{d}^{t}\right)  ,\left(
\varepsilon_{a}^{s+1}-\varepsilon_{c}^{k}\right)  ,-\left(
\varepsilon
_{d}^{t}+\varepsilon_{a}^{s+1}\right)  ,-\left(  \varepsilon_{c}%
^{k}+\varepsilon_{a}^{s+1}\right)\}$  \\
&  &  &  &  &  & \footnotesize{$1\leq a\leq n_{s+1}$}   \\
&  &  &  &  &  & $ \{  \left(  \varepsilon_{a}^{i}-\varepsilon_{d}^{t}\right)
,\left(
\varepsilon_{a}^{i}-\varepsilon_{c}^{k}\right)  ,-\left(  \varepsilon_{a}%
^{i}+\varepsilon_{d}^{t}\right)  ,-\left(
\varepsilon_{a}^{i}+\varepsilon _{c}^{k}\right)\}$   \\
&  &  &  &  &  & \footnotesize{$1\leq i\leq
s,i\neq k,t,1\leq a\leq n_{i}$}  \\
\hline
$\varepsilon_{c}^{k}+\varepsilon_{d}^{s+1}$ &  &  & $\{  \varepsilon_{a}^{k}-\varepsilon_{c}^{k}:\scriptsize{1\leq a\leq
n_{k},a\neq c}\}, \{-\varepsilon_{d}^{s+1}\}$  &  &  & $ \{  \left(  \varepsilon_{a}^{i}-\varepsilon_{c}^{k}\right)
,\left( \varepsilon_{a}^{i}-\varepsilon_{d}^{s+1}\right)  ,-\left(
\varepsilon
_{a}^{i}+\varepsilon_{c}^{k}\right)  ,-\left(  \varepsilon_{a}^{i}%
+\varepsilon_{d}^{s+1}\right)\}$ \\
\scriptsize{$1\leq k\leq s$}&  &  &  &  &  &\footnotesize{$1\leq i\leq s,i\neq k,1\leq a\leq
n_{i}$} \\
 &  &  & $\{  \left( \varepsilon_{a}^{s+1}-\varepsilon_{d}^{s+1}\right)
,-\left( \varepsilon_{a}^{s+1}+\varepsilon_{d}^{s+1}\right)\}$&  &  &  \\
&  &  &\footnotesize{$1\leq a\leq n_{s+1},a\neq d$}  &  &  & $ \{  -\left( \varepsilon_{a}^{k}+\varepsilon_{c}^{k}\right)
,\left( \varepsilon_{a}^{k}-\varepsilon_{d}^{s+1}\right):\scriptsize{1\leq
a\leq n_{k},a\neq c}\}$ \\
&  &  &  &  &  &  \\
 \hline
\end{tabular}
\end{small}
\end{landscape}

\newpage
\begin{landscape}
\begin{small}
\begin{tabular}{ccccccc}
  \hline
  $\alpha\in \Pi_{M}$ &  &  &$\Pi_{\Theta}\left(\alpha\right)$ is the union of &  & &$\Pi_{M}\left(\alpha\right)$ is the union of  \\
  \hline
 &  &  &  &  &  &  \\
 $\varepsilon_{c}
^{k}+\varepsilon_{d}^{k}$&  &  & $\{  \left(
\varepsilon_{a}^{k}-\varepsilon_{c}^{k}\right) ,\left(
\varepsilon_{a}^{k}-\varepsilon_{d}^{k}\right)\}$  &  &  & $\{ \left(  \varepsilon_{a}^{s+1}-\varepsilon_{c}^{k}\right)
,\left( \varepsilon_{a}^{s+1}-\varepsilon_{d}^{k}\right)  ,-\left(
\varepsilon
_{c}^{k}+\varepsilon_{a}^{s+1}\right)  ,-\left(  \varepsilon_{d}
^{k}+\varepsilon_{a}^{s+1}\right)\}$ \\
 &  &  & \footnotesize{$1\leq a\leq n_{k},a\neq c,d$} &  &  &\footnotesize{$1\leq a\leq n_{s+1}$}  \\
\scriptsize{$1\leq k\leq s$} &  &  &  &  &  &  \\
\scriptsize{$1\leq c<d\leq n_{k}$} &  &  &  &  &  &$ \{  \left(  \varepsilon_{a}^{i}-\varepsilon_{c}^{k}\right)
,\left(
\varepsilon_{a}^{i}-\varepsilon_{d}^{k}\right)  ,-\left(  \varepsilon_{a}%
^{i}+\varepsilon_{c}^{k}\right)  ,-\left(
\varepsilon_{a}^{i}+\varepsilon _{d}^{k}\right)\}$ \\
 &  &  &  &  &  & \footnotesize{$1\leq i\leq
s,i\neq k,1\leq a\leq n_{i}$} \\
 &  &  &  &  &  & \\
&  &  &  &  &  &$\{  -\varepsilon_{c}^{k},-\varepsilon_{d}^{k}\}$  \\
  &  &  &  &  &  &  \\
  \hline
 &  &  &  &  &  &  \\
$\varepsilon_{c}^{k}$ &  &  & $\{
\varepsilon_{a}^{k}-\varepsilon_{c}^{k}:\scriptsize{1\leq a\leq n_{k},a\neq c}\}$  &  &  & $\{
\varepsilon_{a}^{k},-\left( \varepsilon_{a}^{k}+\varepsilon
_{c}^{k}\right):\scriptsize{1\leq a\leq n_{k},a\neq c}\}$ \\
 &  &  &  &  &  &  \\
\footnotesize{$1\leq k\leq s$} &  &  & $\{ \pm\varepsilon_{a}^{s+1}:\scriptsize{1\leq a\leq
n_{s+1}}\}$  &  &  & $\{ \pm\varepsilon_{a}^{i},\left(
\varepsilon_{a}^{i}-\varepsilon _{c}^{k}\right)  ,-\left(
\varepsilon_{a}^{i}+\varepsilon_{c}^{k}\right)\}$ \\
 &  &  &  &  &  &\footnotesize{$1\leq i\leq s,i\neq k,1\leq a\leq n_{i}$} \\
 &  &  &  &  &  &  \\
 \hline

 \end{tabular}
 \end{small}
 \end{landscape}

To simplify the algebraic system, we set
$c=\displaystyle\frac{1}{4(2n-1)}$, where $c$ is the Einstein constant such that
$Ric(g)=cg$.
\end{proof}
Note that the previous equations are not symmetric in the sense
that $n_{s+1}$ is treated \\ differently of the others $n_{i}$. In this sense, the algebraic Einstein system is symmetric on the flag manifolds $\mathbb{F}_{B}(n_{1},\ldots,n_{s})$.

\begin{proposition}
\label{Einstein em Bl} \label{Einstein equation Bn 2} The Einstein equations on the spaces
$\mathbb{F}_{B}(n_{1},\ldots,n_{s})$ are given by the following algebraic system

\begin{align*}
&2\left(
n_{k}+n_{t}\right) +\frac{\left(
g_{kt}^{2}-\left(  h_{t}-h_{k}\right)^{2}\right)}{h_{k}h_{t}}
+\sum_{i\neq k,t}^{s}\frac{n_{i}}{g_{ik}
g_{it}}\left(  g_{kt}^{2}-\left(  g_{ik}-g_{it}\right)
^{2}\right)
\\ \\
&+\sum_{i\neq k,t}^{s}\frac{n_{i}}{f_{ik}f_{it}%
}\left(  g_{kt}^{2}-\left(  f_{ik}-f_{it}\right)  ^{2}\right)
+\frac{4\left(  n_{t}-1\right)  }{f_{kt}l_{t}}\left(  g_{kt}
^{2}-\left(  f_{kt}-l_{t}\right)  ^{2}\right)\\ \\
&+\frac{4\left(
n_{k}-1\right)  }{f_{kt}l_{k}}\left( g_{kt}^{2}-\left(
f_{kt}-l_{k}\right) ^{2}\right) =g_{kt}
\end{align*}

\begin{align*}
&2\left(  n_{k}
+n_{t}\right) +\frac{\left(
f_{kt}^{2}-\left(  h_{k}-h_{t}\right)  ^{2}\right)  }{h_{k}h_{t}}+\frac{\left(  n_{k}-1\right)  }{g_{kt}l_{k}
}\left(  f_{kt}^{2}-\left(  g_{kt}-l_{k}\right)  ^{2}\right)\\ \\
&+\frac{\left(  n_{t}-1\right) }{g_{kt}l_{t}}\left(
f_{kt}^{2}-\left(  g_{kt}-l_{t}\right) ^{2}\right)+\sum_{i\neq k,t}^{s}\frac{n_{i}}{f_{ik}%
g_{it}}\left(  f_{kt}^{2}-\left(  f_{ik}-g_{it}\right)
^{2}\right)\\ \\
&+\sum_{i\neq k,t}^{s}\frac{n_{i}}{f_{it}g_{ik}%
}\left(  f_{kt}^{2}-\left(  f_{it}-g_{ik}\right)  ^{2}\right)
=f_{kt}
\end{align*}

\begin{align*}
4\left( n_{k}-1\right)
+\frac{l_{k}^{2}}{h_{k}^{2}}+2\sum_{i\neq k}^{s}\frac{n_{i}}{f_{ik}g_{ik}%
}\left(  l_{k}^{2}-\left(  f_{ik}-g_{ik}\right)  ^{2}\right)
=l_{k}
\end{align*}

\begin{align*}
&2n_{k}
+\sum_{i\neq k}^{s}\frac{n_{i}}{f_{ik}h_{i}%
}\left(  h_{k}^{2}-\left(  f_{ik}-h_{i}\right)  ^{2}\right)
+\frac{\left( n_{k}-1\right)\left( h_{k}^{2}-\left(
h_{k}-l_{k}\right) ^{2}\right) }{h_{k}l_{k}}\\ \\
&+\sum_{i\neq
k}^{s}\frac{n_{i}}{g_{ik}h_{i}}\left( h_{k}^{2}-\left(
g_{ik}-h_{i}\right)  ^{2}\right) =h_{k}
\end{align*}

\end{proposition}

\begin{proof}
This result can be proved analogously to previous way or it can be obtained formally by substituting $n_{s+1}=0$ in the Proposition \ref{equacoes Bl}.
\end{proof}

For $n_{1}=\cdots=n_{s}=1$, in the previous result, we obtain the Einstein equation for full flag manifolds of class $B_n$:
\[
\mathbb{F}_{B}\left(n\right)=SO(2n+1)/U(1)\times\cdots\times
U(1),
\]
where $U(1)$ appears $n$ times and $n\geq2$. These equations were obtained in a similar way by \cite{Sakane}.

\begin{corollary}
The Einstein equations on  $\mathbb{F}_{B}\left(n\right)$
are given by

\begin{align*}
4+\frac{ g_{kt}^{2}-\left(  h_{t}-h_{k}\right)
^{2}  }{h_{k}h_{t}}+\sum_{i\neq k,t}^{n}\frac{ g_{kt}^{2}-\left(  g_{ik}-g_{it}\right)  ^{2}}{g_{ik}g_{it}}+\sum_{i\neq k,t}^{n}\frac{ g_{kt}^{2}-\left(  f_{ik}-f_{it}\right)  ^{2}}{f_{ik}f_{it}} &=& g_{kt}\\ \\
4+\frac{
f_{kt}^{2}-\left(  h_{k}-h_{t}\right) ^{2}  }{h_{k}h_{t}}
+\sum_{i\neq k,t}^{n}\frac{f_{kt}^{2}-\left(  f_{ik}-g_{it}\right)  ^{2}}{f_{ik}g_{it}}+\sum_{i\neq k,t}^{n}\frac{1}{f_{it}g_{ik}%
}  f_{kt}^{2}-\left(  f_{it}-g_{ik}\right)  ^{2} &=& f_{kt}\\ \\
2+\sum_{i\neq
k}^{n}\frac{h_{k}^{2}-\left(  f_{ik}-h_{i}\right)  ^{2}}{f_{ik}h_{i}}   +\sum_{i\neq k}^{n}\frac{ h_{k}^{2}-\left(  g_{ik}-h_{i}\right)  ^{2}}{g_{ik}h_{i}}&=&h_{k}
\end{align*}%
here $g_{ij}=g\left(  E^{ij},E^{ji}\right)$, $f_{ij}=g\left(F^{ij}
,F^{-ij}\right)  ,$ $1\leq i<j\leq n$ and $h_{i}=g\left(
G^{i},G^{-i}\right)$, $1\leq i\leq n.$
\end{corollary}

\subsection{Case $C_{n}$}

The flag manifolds of type $C_n$ have one of the following forms

\begin{align*}
\mathbb{F}_{C}(n_{1},\ldots,n_{s}) &= Sp(n)/ U\left(
n_{1}\right) \times\cdots\times U\left( n_{s}\right),\\ \\
\mathbb{F}_{C}[n_{1},\ldots,n_{s+1}] &= Sp(n)/\left(  U\left(  n_{1}\right)  \times\cdots\times U\left(n_{s}\right)  \times Sp\left(  n_{s+1}\right)  \right)
\end{align*}
where $n=\sum n_{i}$ and  $n_{s+1}>1$.
\begin{theorem}\label{t-roots properties Cn}
The set $\Pi_{\mathfrak{t}}$ of t-roots corresponding to the flag manifolds $\mathbb{F}_{C}(n_{1},\ldots,n_{s})$ is a system of roots of type $C_{s}$.
\end{theorem}

\begin{proof}
A Cartan subalgebra of $\mathfrak{sp}(n,\mathbb{C})$ consists in
taking matrices of the form
\begin{equation}
\mf{h}=
\begin{pmatrix}
\Lambda & 0\\
0 & -\Lambda
\end{pmatrix}\label{alg cartan Cl}
\end{equation}
where $\Lambda=\mathrm{diag}(\varepsilon_{1},\ldots,\varepsilon_{n})$, $\varepsilon_{i}%
\in\mathbb{C}$. Thus the root system is
\begin{equation}
\Pi=\{  \pm\left(  \varepsilon_{i}-\varepsilon_{j}\right)
,\pm\left( \varepsilon_{i}+\varepsilon_{j}\right)  ;\,1\leq i<j\leq
n\}  \cup\{ \pm2\varepsilon_{i};\,1\leq i\leq
n\}.\label{Pi Cl}
\end{equation}
The root system for the subalgebra $\mathfrak{k}_{\Theta}^{\mathbb{C}
}=\mathfrak{sl}\left(  n_{1},\mathbb{C}\right)
\times\cdots\times\mathfrak{sl}\left( n_{s},\mathbb{C}\right)$ is
given by
\[
\Pi_{\Theta} =\{  \pm\left(  \varepsilon_{a}^{i}-\varepsilon_{b}
^{i}\right)  ;\,1\leq a<b\leq n_{i},1\leq i\leq s \}.
\]
Then
\[
\Pi_{M}=\{  \pm (  \varepsilon_{a}^{i}-\varepsilon_{b}
^{j})  ,\pm (
\varepsilon_{a}^{i}+\varepsilon_{b}^{j});\,1\leq i<j\leq s\} \\
\cup\{  \pm\left(
\varepsilon_{a}^{i}+\varepsilon_{b}^{i}\right) ;\,1\leq i\leq
s,1\leq a\leq b\leq n_{i}\}.
\]
We see that the algebra $\mf{t}$ has the form
\[
\mf{t}=
\begin{pmatrix}
\Lambda & 0\\
0 & -\Lambda
\end{pmatrix}
\]
with
\[
\Lambda
=\mathrm{diag}\left(\varepsilon^{1}_{n_{1}},\ldots,\varepsilon^{1}_{n_{1}},\varepsilon^{2}_{n_{2}},\ldots,\varepsilon^{3}_{n_{3}},\ldots,
\varepsilon^{s}_{n_{s}},\ldots,\varepsilon^{s}_{n_{s}}\right)
\]
where each $\varepsilon^{i}_{n_{i}}$ appears exactly $n_{i}$
times, $i=1,\ldots,s$. So restricting the roots of
$\Pi_{M}$ in $\mf{t}$, and using the notation
$\delta_{i}=k(\varepsilon^{i}_{a})$, we obtain the t-root set:
\[
\Pi_{\mathfrak{t}}=\{\pm\left(\delta_{i}-\delta_{j}\right),\pm\left(\delta_{i}+\delta_{j}\right);\,1\leq
i<j\leq s\}\cup\{\pm 2\delta_{i};\,1\leq i\leq
s\}.
\]
Note that
$k(\varepsilon_{a}^{i}+\varepsilon_{b}^{i})
=k(2\varepsilon_{a}^{i})$, $ 1\leq i\leq s$. So the number of positive t-roots is $s^{2}$ .
\end{proof}

The Killing form of
$\mathfrak{sp(n)}$ is
$$B\left( X,Y\right) =2\left( n+1\right)\mathrm{tr}(XY),$$ 
and
$$ B( \alpha,\alpha)=\left\{
\begin{array}
[c]{cc}
1/(n+1), & \text{if} \quad \alpha=\pm2\varepsilon_{i},\\
1/2(n+1), & \text{if} \quad \alpha= \pm ( \varepsilon
_{i}\pm\varepsilon_{j}),\quad 1\leq i<j\leq
n.
\end{array}
\right.
$$
Then the eigenvectors $X_{\alpha}\in \mf{g}_{\alpha}$ satisfying (\ref{base weyl})  are
\begin{align*}
&X_{\pm\left(\varepsilon_{i}-\varepsilon_{j}\right)}=\pm\frac
{1}{2\sqrt{n+1}}E_{\pm\left(
\varepsilon_{i}-\varepsilon_{j}\right)},\quad X_{\pm\left(
\varepsilon_{i}+\varepsilon_{j}\right)}=\pm\frac{1}
{2\sqrt{n+1}}E_{\pm\left(  \varepsilon_{i}+\varepsilon_{j}\right)
},\quad 1\leq i<j\leq n;&\\ \\&
X_{\pm2\varepsilon_{i}}=\pm\frac{1}{\sqrt{2\left(  n+1\right)  }}
E_{\pm2\varepsilon_{i}},\quad 1\leq i\leq n,
\end{align*}
where $E_{\alpha}$ denotes the canonical eigenvectors of $\mf{g}_{\alpha}$. As in the previous case we use the convenient notation
\begin{align*}
E_{ab}^{ij}  &  =X_{\varepsilon_{a}^{i}-\varepsilon_{b}^{j}},\quad
F_{ab}^{ij}=X_{\varepsilon_{a}^{i}+\varepsilon_{b}^{j}},\ F_{-ab}^{ij}=X_{-\left(\varepsilon_{a}^{i}+\varepsilon_{b}^{j}\right)
},\quad 1\leq i<j\leq s;\\
F_{ab}^{i}  &  =X_{\varepsilon_{a}^{i}+\varepsilon_{b}^{i}},\ F_{-ab}^{-i}=X_{-\left(  \varepsilon_{a}^{i}+\varepsilon_{b}^{i}\right)
},\quad 1\leq
a\neq b\leq n_{i};\\ G_{a}^{i} & =X_{2\varepsilon_{a}^{i}},\quad
G_{-a}^{-i}=X_{-2\varepsilon_{a}^{i}},\text{ }1\leq i\leq s.\nonumber
\end{align*}
An invariant metric on
$\mathbb{F}_{C}(n_{1},\ldots,n_{s})$ will be denoted by
\begin{align}\label{inv metric Cl}
g_{ij}  &  =g\left(  E_{ab}^{ij},E_{ba}^{ji}\right), \quad f_{ij}=g\left(
F_{ab}^{ij},F_{-ab}^{-ij}\right)  ,\quad 1\leq i<j\leq s;\\
h_{i}  &  =g\left(  G_{a}^{i},G_{-a}^{-i}\right),\quad l_{i}=g\left( F_{ab}^{i},F_{-ab}^{-i}\right)  ,\quad 1\leq i\leq s.\nonumber
\end{align}
Since
$k(\varepsilon_{a}^{i}+\varepsilon_{b}^{i})
=k(2\varepsilon_{a}^{i})=2k(\varepsilon_{a}^{i})$ it follows
\[
l_{i}=h_{i},\quad 1\leq i\leq s,
\]
by remark \ref{remark}.

Considering short and long roots of $\mathfrak{sp}(n)$, one can see that the square of structural constants are given by
\begin{align*}
& N_{\left(  \varepsilon_{i}+\varepsilon_{j}\right)  ,\pm\left(
\varepsilon _{i}-\varepsilon_{j}\right)
}^{2}=N_{\pm2\varepsilon_{j},\left(
\varepsilon_{i}\mp\varepsilon_{j}\right)
}^{2}=N_{-2\varepsilon_{i},\left(
\varepsilon_{i}\pm\varepsilon_{j}\right)  }^{2}=\frac{1}{2\left(
n+1\right) };\quad i\neq j\\
& N_{\left(  \varepsilon_{i}-\varepsilon_{j}\right)
,\alpha}^{2}=N_{\left( \varepsilon_{i}+\varepsilon_{j}\right)
,\beta}^{2}=\frac{1}{4\left( n+1\right)}
\end{align*}
if $\alpha\in\{  \left(
\varepsilon_{k}-\varepsilon_{i}\right)  ,\left(
\varepsilon_{j}-\varepsilon_{k}\right)  ,\left(  \varepsilon_{j}
+\varepsilon_{l}\right)  ,-\left(
\varepsilon_{i}+\varepsilon_{p}\right): \,p\neq i;l\neq
j; k\neq i,j; i\neq j\}$ and
$\beta\in\{  -\left(  \varepsilon_{i}+\varepsilon_{k}\right)
,-\left(  \varepsilon_{j}+\varepsilon_{l}\right):\, k\neq
j;l\neq i\}$.

In the next table we compute $\Pi_{\Theta}(\alpha)$ and $\Pi_M(\alpha)$ for each $\alpha\in \Pi_M$.

\begin{landscape}
\begin{small}
\begin{tabular}{ccccccc}
  \hline
  $\alpha\in \Pi_{M}$ &  &  &$\Pi_{\Theta}\left(\alpha\right)$ is the union of &  & &$\Pi_{M}\left(\alpha\right)$ is the union of  \\
  \hline
 &  &  &  &  &  &  \\
$\varepsilon_{c} ^{k}-\varepsilon_{d}^{t}$ &  &  &$\{ \varepsilon_{a}^{k}-\varepsilon_{c}^{k}:\scriptsize{1\leq a\leq
n_{k},a\neq c}\}$  &  &  & $ \{  \left(  \varepsilon_{d}^{t}-\varepsilon_{a}^{i}\right)
,\left(
\varepsilon_{a}^{i}-\varepsilon_{c}^{k}\right)  ,\left(  \varepsilon_{a}%
^{i}+\varepsilon_{d}^{t}\right)  ,-\left(
\varepsilon_{a}^{i}+\varepsilon _{c}^{k}\right)\}$\\
\scriptsize{$1\leq k<t\leq s$}  &  & &$\{
\varepsilon_{d}^{t}-\varepsilon_{a}^{t}:\scriptsize{1\leq a\leq n_{t},a\neq
d}\} $   &  &  & \footnotesize{$1\leq i\leq
s,i\neq k,t\quad\text{and}\quad 1\leq a\leq n_{i}$} \\
&  &  &  &  &  & \\
 &  &  &  &  &  & $\{  \left(  \varepsilon_{a}^{k}+\varepsilon_{d}^{t}\right)
,-\left( \varepsilon_{a}^{k}+\varepsilon_{c}^{k}\right):\scriptsize{1\leq
a\leq n_{k}}\}$   \\
&  &  &  &  &  &$\{  -\left(
\varepsilon_{a}^{t}+\varepsilon_{c}^{k}\right)  ,\left(
\varepsilon_{a}^{t}+\varepsilon_{d}^{t}\right):\scriptsize{1\leq a\leq
n_{t}}\}$  \\
 \hline
 &  &  &  &  &  &  \\
$\varepsilon_{c}^{k}+\varepsilon_{d}^{t}$ &  &  &$\{ \varepsilon_{a}^{k}-\varepsilon_{c}^{k}:\scriptsize{1\leq a\leq
n_{k},a\neq c}\}$ &  &  & $ \{  \left(  \varepsilon_{a}^{i}-\varepsilon_{c}^{k}\right)
,\left(
\varepsilon_{a}^{i}-\varepsilon_{d}^{t}\right)  ,-\left(  \varepsilon_{a}%
^{i}+\varepsilon_{c}^{k}\right)  ,-\left(
\varepsilon_{a}^{i}+\varepsilon
_{d}^{t}\right)\}$  \\
\scriptsize{$1\leq k<t\leq s$} &  &  &$\{
\varepsilon_{a}^{t}-\varepsilon_{d}^{t}:\scriptsize{1\leq a\leq n_{t},a\neq
d}\}$  &  &  & \footnotesize{$ 1\leq i\leq s,i\neq k,t;1\leq a\leq n_{i}$} \\
 &  &  &   &  &  &  \\
 &  &  &  &  &  & $ \{  \left(  \varepsilon_{a}^{t}-\varepsilon_{c}^{k}\right),-\left( \varepsilon_{a}^{t}+\varepsilon_{d}^{t}\right):\scriptsize{1\leq a\leq n_{t}}\}$  \\
 &  &  &  &  &  & $\{  \left(
\varepsilon_{a}^{k}-\varepsilon_{d}^{t}\right)  ,-\left(
\varepsilon_{a}^{k}+\varepsilon_{c}^{k}\right):\scriptsize{1\leq a\leq
n_{k}}\} $  \\
 \hline
 &  &  &  &  &  &  \\
$\varepsilon_{c}^{k}+\varepsilon_{d}^{k}$ &  &  & $\{ \left(  \varepsilon_{a}^{k}-\varepsilon_{c}^{k}\right)
,\left( \varepsilon_{a}^{k}-\varepsilon_{d}^{k}\right)\}$  &  &  & $\{ \left(  \varepsilon_{a}^{i}-\varepsilon_{c}^{k}\right)
,\left(
\varepsilon_{a}^{i}-\varepsilon_{d}^{k}\right)  ,-\left(  \varepsilon_{a}%
^{i}+\varepsilon_{c}^{k}\right)  ,-\left(
\varepsilon_{a}^{i}+\varepsilon _{d}^{k}\right)\}$ \\
\scriptsize{$1\leq k\leq s$} &  &  &\footnotesize{$1\leq
a\leq n_{k};a\neq
c,d$}  &  &  &\footnotesize{$1\leq i\leq
s;i\neq k$}  \\
\scriptsize{$1\leq c<d\leq n_{k}$} &  &  &  &  &  &  \\
 &  &  & $\{  \pm\left(  \varepsilon_{c}^{k}-\varepsilon_{d}^{k}\right)  \}$  &  &  &  \\
 \hline
&  &  &  &  &  &  \\
$2\varepsilon_{c}^{k}$&  &  & $\{  \varepsilon_{a}
^{k}-\varepsilon_{c}^{k}:\scriptsize{1\leq a\leq n_{k};a\neq c}\}$ &  &  &  $\{  \left(
\varepsilon
_{a}^{i}-\varepsilon_{c}^{k}\right)  ,-\left(  \varepsilon_{a}^{i}%
+\varepsilon_{c}^{k}\right):\scriptsize{1\leq i\leq s;i\neq k}\}$\\
\scriptsize{$1\leq k\leq s$}&  &  &  &  &  &  \\
&  &  &  &  &  &  \\
\hline
 \end{tabular}
 \end{small}
 \end{landscape}

Now, following proposition \ref{Ricci} we obtain the following result.

\begin{proposition}\label{Einstein equation Cn 1}
The Einstein equation for an invariant metric on
$\mathbb{F}_{C}(n_{1},\ldots,n_{s})$ reduces to an
algebraic system where the number of unknowns and equations is $s^2$, given by

\begin{align*}
&2(n_{k}+n_{t})+\frac{( n_{k}+1) }{h_{k}f_{kt}}\left(
g_{kt}^{2}-\left( h_{k}-f_{kt}\right)  ^{2}\right)+\frac{\left(  n_{t}+1\right)  }{h_{t}f_{kt}}\left(  g_{kt}^{2}-\left(  h_{t}-f_{kt}\right)  ^{2}\right)\\ \\
&+\sum_{i\neq k,t}^{s}\frac{n_{i}}{g_{ik}g_{it}
}\left(  g_{kt}^{2}-\left(  g_{ik}-g_{it}\right)  ^{2}\right)+
\sum_{i\neq k,t}^{s}\frac{n_{i}}{f_{ik}f_{it}
}\left(  g_{kt}^{2}-\left(  f_{ik}-f_{it}\right)  ^{2}\right)
=g_{kt},\,\quad 1\leq k\neq t\leq s;
\end{align*}

\begin{align*}
&2(n_{k}+n_{t})+\frac{\left( n_{k}+1\right) }{h_{k}g_{kt}}\left(
f_{kt}^{2}-\left( h_{k}-g_{kt}\right)  ^{2}\right)
+\frac{\left(  n_{t}+1\right)  }{h_{t}g_{kt}%
}\left(  f_{kt}^{2}-\left(  h_{t}-g_{kt}\right)  ^{2}\right)\\ \\
&+\sum_{i\neq k,t}^{s}\frac{n_{i}}{f_{it}g_{ik}%
}\left(  f_{kt}^{2}-\left(  f_{it}-g_{ik}\right)  ^{2}\right)
+\sum_{i\neq k,t}^{s}\frac{n_{i}}{f_{ik}g_{it}
}\left(  f_{kt}^{2}-\left(  f_{ik}-g_{it}\right)  ^{2}\right)
=f_{kt},\quad 1\leq k\neq t\leq s;\\ \\
&4(n_{k}+1)+2\sum_{i\neq k}^{s}\frac{n_{i}}{f_{ik}g_{ik}%
}\left(  h_{k}^{2}-\left(  f_{ik}-g_{ik}\right)  ^{2}\right)=h_{k},\quad 1\leq k\leq s.
\end{align*}

\end{proposition}

If each $n_{i}=1$ then $s=n$ and thus we obtain the full
flag manifolds of type $C_{n}$, which will be denoted by
\[
\mathbb{F}_{C}(n)=Sp(n)/U\left( 1\right)\times\cdots\times
U\left( 1\right)
\]
where $U(1)$ appears $n$ times. It is easy to apply the previous result to this
space, which were treated by \cite{Sakane}.

Now we will discuss the space $\mathbb{F}_{C}[n_{1},\ldots,n_{s+1}]$. Considering Cartan subalgebra $\mathfrak{h}$ of $\mathfrak{sp}(n,\mathbb{C})$, given in (\ref{alg cartan Cl}), we see that the root
system of the pair $(\mathfrak{k}_{\Theta}^{\mathbb{C}},\mathfrak{h})$,
with  $\mathfrak{k}_{\Theta}^{\mathbb{C}}=\mathfrak{sl}\left(  n_{1},\mathbb{C}\right)\times\cdots\times\mathfrak{sl}\left( n_{s},\mathbb{C}\right)\times \mf{sp}\left( n_{s+1},\mathbb{C}\right)$ is given by

\begin{align*}
\Pi_{\Theta}  &  =\{  \pm\left(  \varepsilon_{a}^{i}-\varepsilon_{b}
^{i}\right)  ;\,1\leq i\leq s+1\}  \cup\{  \pm\left(  \varepsilon
_{a}^{s+1}+\varepsilon_{b}^{s+1}\right)  ;\,1\leq a<b\leq n_{s+1}\} \\
&  \cup\{  \pm2\varepsilon_{a}^{s+1};\,1\leq a\leq n_{s+1}\}.
\end{align*}
Thus
\begin{align*}
\Pi_{M}  &  =\{  \pm(  \varepsilon_{a}^{i}-\varepsilon_{b}^{j})  ,\pm(  \varepsilon_{a}^{i}+\varepsilon_{b}^{j})
;\,1\leq i<j\leq s+1\} \\
&  \cup\{  \pm(  \varepsilon_{a}^{i}+\varepsilon_{b}^{i});\,1\leq i\leq s,1\leq a\leq b\leq n_{i}\}.
\end{align*}

In the next table we compute the sets $\Pi_{\Theta}(\alpha)$ and $\Pi_M(\alpha)$ for each $\alpha\in\Pi_M^+$.
\begin{landscape}
\begin{tabular}{ccccccc}
  \hline
  $\alpha\in \Pi_{M}$ &  &  &$\Pi_{\Theta}\left(\alpha\right)$ is the union of &  & &$\Pi_{M}\left(\alpha\right)$ is the union of  \\
  \hline
&  &  &  &  &  &  \\
$\varepsilon_{c}^{k}-\varepsilon_{d}^{t}$&  &  &$\{
\varepsilon_{a}^{k}-\varepsilon_{c}^{k}\}$   &  &  &$ \{  \left(  \varepsilon_{d}^{t}-\varepsilon_{a}^{i}\right)  ,\left(
\varepsilon_{a}^{i}-\varepsilon_{c}^{k}\right)  ,\left(  \varepsilon_{a}%
^{i}+\varepsilon_{d}^{t}\right)  ,-\left(  \varepsilon_{a}^{i}+\varepsilon
_{c}^{k}\right)\}$  \\
\scriptsize{$1\leq k<t\leq s$} &  &  &\footnotesize{$1\leq a\leq n_{k},a\neq c$}  &  &  &\footnotesize{$1\leq i\leq s+1,i\neq k,t;1\leq a\leq
n_{i} $ }  \\
&  &  & $\{  \varepsilon_{d}^{t}-\varepsilon_{a}^{t}\}$   &  &  & $\{  \left(  \varepsilon_{a}^{k}+\varepsilon_{d}^{t}\right)  ,-\left(
\varepsilon_{a}^{k}+\varepsilon_{c}^{k}\right):\scriptsize{1\leq a\leq n_{k}}\}$  \\
&  &  &\footnotesize{$1\leq a\leq n_{t},a\neq
d$}  &  &  &$
\{  -\left(  \varepsilon_{a}^{t}+\varepsilon_{c}^{k}\right)  ,\left(\varepsilon_{a}^{t}+\varepsilon_{d}^{t}\right):\scriptsize{1\leq a\leq n_{t}}\} $  \\
\hline
&  &  &  &  &  &  \\
$\varepsilon_{c}^{k}-\varepsilon
_{d}^{s+1}$&  &  & $\{  \left(  \varepsilon_{d}^{s+1}-\varepsilon_{a}^{s+1}\right)  ,\left(
\varepsilon_{a}^{s+1}+\varepsilon_{d}^{s+1}\right)\}$  &  &  & $ \{  \left(  \varepsilon_{d}^{s+1}-\varepsilon_{a}^{i}\right)  ,\left(
\varepsilon_{a}^{i}-\varepsilon_{c}^{k}\right)  ,\left(  \varepsilon_{a}^{i}+\varepsilon_{d}^{s+1}\right)  ,-\left(  \varepsilon_{a}^{i}+\varepsilon_{c}^{k}\right)\}$   \\
\footnotesize{$1\leq k\leq s$}&  &  &\footnotesize{$1\leq a\leq n_{s+1},a\neq
d$}  &  &  &\footnotesize{$1\leq i\leq s,i\neq k;1\leq a\leq n_{i}$}  \\
&  &  & $\{  \varepsilon_{a}^{k}-\varepsilon_{c}^{k}\}$   &  &  & $ \{  \left(  \varepsilon_{a}^{k}+\varepsilon_{d}^{s+1}\right),-\left(  \varepsilon_{a}^{k}+\varepsilon_{c}^{k}\right)  \} $  \\
&  &  &\footnotesize{$1\leq a\leq
n_{k},a\neq c $}  &  &  &  \\
\hline
&  &  &  &  &  &  \\
$\varepsilon_{c}^{k}+\varepsilon_{d}^{t}$&  &  & $\{
\varepsilon_{a}^{k}-\varepsilon_{c}^{k}\}$  &  &  & $\{  \left(  \varepsilon_{a}^{i}-\varepsilon_{c}^{k}\right)  ,\left(
\varepsilon_{a}^{i}-\varepsilon_{d}^{t}\right)  ,-\left(  \varepsilon_{a}%
^{i}+\varepsilon_{c}^{k}\right)  ,-\left(  \varepsilon_{a}^{i}+\varepsilon
_{d}^{t}\right)\}$  \\
\footnotesize{$1\leq k<t\leq s$}&  &  &\footnotesize{$1\leq a\leq n_{k},a\neq c$}  &  &  &\footnotesize{$1\leq i\leq s+1,i\neq k,t;1\leq a\leq n_{i}$} \\
&  &  &$\{  \varepsilon_{a}^{t}-\varepsilon_{d}^{t}\}$  &  &  &$ \{  \left(  \varepsilon_{a}^{t}-\varepsilon_{c}^{k}\right)  ,-\left(
\varepsilon_{a}^{t}+\varepsilon_{d}^{t}\right):\scriptsize{1\leq a\leq n_{t}}\}$  \\
&  &  &\footnotesize{$1\leq a\leq n_{t},a\neq
d$}   &  &  &$\{  \left(  \varepsilon_{a}^{k}-\varepsilon_{d}^{t}\right)  ,-\left(
\varepsilon_{a}^{k}+\varepsilon_{c}^{k}\right):\scriptsize{1\leq a\leq n_{k}}\}$  \\
\hline
 \end{tabular}
 \end{landscape}
\newpage
\begin{landscape}
\begin{tabular}{ccccccc}
  \hline
  $\alpha\in \Pi_{M}$ &  &  &$\Pi_{\Theta}\left(\alpha\right)$ is the union of &  & &$\Pi_{M}\left(\alpha\right)$ is the union of  \\
  \hline
&  &  &  &  &  &  \\
$\varepsilon_{c}^{k}+\varepsilon_{d}^{s+1}$&  &  & $\{\varepsilon_{a}^{k}-\varepsilon_{c}^{k}\}$  &  &  & $ \{  \left(  \varepsilon_{a}^{i}-\varepsilon_{c}^{k}\right)  ,\left(
\varepsilon_{a}^{i}-\varepsilon_{d}^{s+1}\right)  ,-\left(  \varepsilon
_{a}^{i}+\varepsilon_{d}^{s+1}\right)  ,-\left(  \varepsilon_{a}
^{i}+\varepsilon_{c}^{k}\right)\}$  \\
\footnotesize{$1\leq k\leq s$}&  &  &\footnotesize{$1\leq a\leq n_{k},a\neq
c$}  &  &  &\footnotesize{$1\leq i\leq s,i\neq k;1\leq a\leq
n_{i} $}  \\
&  &  &$ \{  \left(  \varepsilon_{a}^{s+1}-\varepsilon_{d}^{s+1}\right)
,-\left(  \varepsilon_{a}^{s+1}+\varepsilon_{d}^{s+1}\right)  ,-2\varepsilon
_{d}^{s+1}\}$  &  &  & $ \{  \left(  \varepsilon_{a}^{k}-\varepsilon_{d}^{s+1}\right)
,-\left(  \varepsilon_{a}^{k}+\varepsilon_{c}^{k}\right)\}$  \\
&  &  &\footnotesize{$1\leq a\leq n_{s+1},a\neq d$}  &  &  &\footnotesize{$1\leq a\leq
n_{k} $}  \\
\hline
&  &  & &  &  &  \\
$\varepsilon_{c}^{k}+\varepsilon_{d}^{k}$&  &  &$\{
\left(  \varepsilon_{a}^{k}-\varepsilon_{c}^{k}\right)  ,\left(
\varepsilon_{a}^{k}-\varepsilon_{d}^{k}\right)\}$ &  &  & $\{
\left(  \varepsilon_{a}^{i}-\varepsilon_{c}^{k}\right)  ,\left(
\varepsilon_{a}^{i}-\varepsilon_{d}^{k}\right)  ,-\left(  \varepsilon_{a}%
^{i}+\varepsilon_{c}^{k}\right)  ,-\left(  \varepsilon_{a}^{i}+\varepsilon
_{d}^{k}\right)\} $  \\
\footnotesize{$1\leq k\leq
s$}&  &  &\footnotesize{ $1\leq a\leq n_{k};a\neq
c,d $  } &  &  &\footnotesize{$1\leq i\leq s+1;i\neq k$}  \\
\footnotesize{$1\leq c<d\leq n_{k}$}&  &  &$\{  \pm\left(  \varepsilon_{c}^{k}-\varepsilon_{d}%
^{k}\right)  \}$ &  &  &  \\
\hline
&  &  & &  &  &  \\
$2\varepsilon_{c}^{k}$&  &  & $\{  \varepsilon_{a}
^{k}-\varepsilon_{c}^{k}\}$ &  &  &  $\{  \left(  \varepsilon
_{a}^{i}-\varepsilon_{c}^{k}\right)  ,-\left(  \varepsilon_{a}^{i}
+\varepsilon_{c}^{k}\right)\}$\\
\footnotesize{$1\leq k\leq s$}&  &  &\footnotesize{$1\leq a\leq n_{k};a\neq c$} &  &  & \footnotesize{$1\leq i\leq s+1;i\neq k$} \\
\hline
 \end{tabular}
 \end{landscape}

The subalgebra $\mathfrak{t}$ is given by
\[
\mathfrak{t}=
\begin{pmatrix}
\Lambda & 0\\
0 & -\Lambda
\end{pmatrix}
\]
with
\[
\Lambda=\mathrm{diag}(\varepsilon^{1}_{n_{1}},\ldots,\varepsilon^{1}_{n_{1}},\ldots,\varepsilon^{s}_{n_{s}},\ldots,\varepsilon^{s}_{n_{s}},0,\ldots,0),
\]
where $\varepsilon^{i}_{n_{i}}$ appears $n_{i}$ times for $i=1,\ldots,s$ and $0$ appears $n_{s+1}$ times. By restricting the roots of $\Pi_{M}^{+}$ on
$\mathfrak{t}$, we see that the t-root set is
\[
\Pi_{\mathfrak{t}}=\{  \pm\left(  \delta_{i}\pm \delta_{j}\right)  ;\,1\leq i<j\leq
s\}  \cup\{  \pm \delta_{i},\pm2\delta_{i};\,1\leq i\leq s\}
\]
where $\delta_{i}=k(\varepsilon^{i}_{a})$. Note that the number of positive t-roots is $s^{2}+s$. We keep the notation (\ref{inv metric Cl}) for an invariant metric on $\mathbb{F}_{C}[n_{1},\ldots,n_{s+1}]$, but here we note that
\[
g_{i(s+1)}=f_{i(s+1)}\quad \text{and}\quad l_{i}=h_{i},\quad 1\leq i\leq s,
\]
because
$k(\varepsilon_{a}^{i}-\varepsilon_{b}^{s+1}
)=k(\varepsilon_{a}^{i}-\varepsilon_{b}^{s+1})$ and $k(\varepsilon_{a}
^{i}+\varepsilon_{b}^{i})=k(2\varepsilon_{a}^{i})$ for $1\leq i\leq s$.
Thus we obtain
\begin{proposition} \label{Einstein equation Cn 2}
The Einstein equation on the space $\mathbb{F}_{C}[n_{1},\ldots,n_{s+1}]$ reduces to an algebraic system where the number of unknowns and equations is $s^2+s$, given by

\begin{align*}
&\frac{n_{k}+n_{t}}{4\left(
n+1\right)  }+\frac{1}{16\left(  n+1\right)  }\frac{\left(  2n_{k}+1\right)
}{h_{k}f_{kt}}\left(  g_{kt}^{2}-\left(  h_{k}-f_{kt}\right)  ^{2}\right)+\frac{1}{16\left(  n+1\right)  }\frac{\left(  2n_{t}+1\right)  }{h_{t}f_{kt}
}\left(  g_{kt}^{2}-\left(  h_{t}-f_{kt}\right)  ^{2}\right)\\ \\
& +\frac
{1}{8\left(  n+1\right)  }\sum_{i\neq k,t}^{s}\frac{n_{i}}{g_{ik}g_{it}%
}\left(  g_{kt}^{2}-\left(  g_{ik}-g_{it}\right)  ^{2}\right)
+\frac{1}{8\left(  n+1\right)  }\sum_{i\neq k,t}^{s}\frac{n_{i}}{f_{ik}f_{it}%
}\left(  g_{kt}^{2}-\left(  f_{ik}-f_{it}\right)  ^{2}\right)
\\ \\
&+\frac
{1}{4\left(  n+1\right)  }\frac{n_{s+1}}{g_{k(s+1)}g_{t(s+1)}}\left(
g_{kt}^{2}-\left(  g_{k(s+1)-}g_{t(s+1)}\right)  ^{2}\right)=cg_{kt},\quad 1\leq k\neq t\leq s;
\end{align*}
\\
\begin{align*}
&\frac{n_{k}+2n_{s+1}%
+1}{4\left(  n+1\right)  }+\frac{1}{16\left(  n+1\right)  }\frac{\left(
2n_{k}+1\right)  }{h_{k}g_{k\left(  s+1\right)  }}\left(  g_{k\left(
s+1\right)  }^{2}-\left(  h_{k}-g_{k\left(  s+1\right)  }\right)  ^{2}\right)\\ \\
&+\frac{1}{8\left(  n+1\right)  }\sum_{i\neq k}^{s}\frac{n_{i}}%
{g_{ik}g_{i\left(  s+1\right)  }}\left(  g_{k\left(  s+1\right)  }^{2}-\left(
g_{ik}-g_{i\left(  s+1\right)  }\right)  ^{2}\right) \\ \\
&  +\frac{1}{8\left(  n+1\right)  }\sum_{i\neq k}^{s}\frac{n_{i}}%
{f_{ik}g_{i\left(  s+1\right)  }}\left(  g_{k\left(  s+1\right)  }^{2}-\left(
f_{ik}-g_{i\left(  s+1\right)  }\right)  ^{2}\right)=cg_{k(s+1)},\quad 1\leq k\leq s;
\end{align*}

\begin{align*}
&\frac{n_{k}+n_{t}}{4\left(
n+1\right)  }+\frac{1}{16\left(  n+1\right)  }\frac{\left(  2n_{k}+1\right)
}{h_{k}g_{kt}}\left(  f_{kt}^{2}-\left(  h_{k}-g_{kt}\right)  ^{2}\right)+\frac{1}{16\left(  n+1\right)  }\frac{\left(  2n_{t}+1\right)  }{h_{t}g_{kt}%
}\left(  f_{kt}^{2}-\left(  h_{t}-g_{kt}\right)  ^{2}\right)\\ \\
&  +\frac
{1}{8\left(  n+1\right)  }\sum_{i\neq k,t}^{s}\frac{n_{i}}{f_{it}g_{ik}%
}\left(  f_{kt}^{2}-\left(  f_{it}-g_{ik}\right)  ^{2}\right)
+\frac{1}{8\left(  n+1\right)  }\sum_{i\neq k,t}^{s}\frac{n_{i}}{f_{ik}g_{it}%
}\left(  f_{kt}^{2}-\left(  f_{ik}-g_{it}\right)  ^{2}\right)\\ \\
&  +\frac
{1}{4\left(  n+1\right)  }\frac{n_{s+1}}{g_{k(s+1)}g_{t(s+1)}}\left(
f_{kt}^{2}-\left(  g_{k(s+1)-}g_{t(s+1)}\right)  ^{2}\right)=cf_{kt},\quad 1\leq k\neq t\leq s;
\end{align*}

\begin{align*}
&\frac{n_{k}+2n_{s+1}+1}{4\left(  n+1\right)  }+\frac{1}{16\left(
n+1\right)  }\frac{\left(  2n_{k}+1\right)  }{h_{k}g_{k(s+1)}}\left(
f_{k(s+1)}^{2}-\left(  h_{k}-g_{k(s+1)}\right)  ^{2}\right)\\ \\
&+\frac{1}{8\left(  n+1\right)  }\sum_{i\neq k}^{s}\frac{n_{i}}{g_{ik}%
g_{i(s+1)}}\left(  f_{k(s+1)}^{2}-\left(  g_{ik}-g_{i(s+1)}\right)^{2}\right)=
cg_{k\left(  s+1\right)  },\quad 1\leq k\leq s;
\end{align*}

\begin{align*}
\frac{n_{k}+1}{2\left(  n+1\right)  }+\frac{n_{s+1}}{4\left(
n+1\right)  }\frac{h_{k}^{2}}{f_{k(s+1)}^{2}}
+\frac{1}{4\left(  n+1\right)  }\sum_{i\neq k}^{s}\frac{n_{i}}{f_{ik}g_{ik}}\left(  h_{k}^{2}-\left(  f_{ik}-g_{ik}\right)  ^{2}\right)=ch_{k},\quad 1\leq k\leq s.
\end{align*}
\end{proposition}

\subsection{Case $D_{n}$}

Flag manifolds of type $D_n$ are spaces of the form
\begin{align*}
&\mathbb{F}_{D}(n_{1},\ldots, n_{s})=
SO(2n)/U(n_{1})\times\cdots\times U(n_{s}), \\ \\
&\mathbb{F}_{D}[n_{1},\ldots, n_{s+1}]=SO(2n)/U(n_{1})\times\cdots\times U(n_{s})\times SO(2n_{s+1}),
\end{align*}
where $n=\sum n_{i}$ and $n_{s+1}\geq4$.

The matrices in the algebra $\mathfrak{so}\left( 2n,\mathbb{C}\right)$ of the skew-symmetric matrices in even dimension can be written as
\[
A=\left(\begin{array}{cc}
\alpha & \beta \\
\gamma & -\alpha^t
\end{array} \right)
\]
where $\alpha,\beta,\gamma$ are matrices $n\times n$ with $\beta,\gamma$ skew-symmetric.
\begin{theorem} \label{t-roots properties Dn}
The set of t-roots $\Pi_{\mathfrak{t}}$ corresponding to the space
$\mathbb{F}_{D}(n_{1},\ldots, n_{s})$ is a root system of type
$C_{s}$.
\end{theorem}
\begin{proof}
A Cartan subalgebra of $\mathfrak{so}\left( 2n,\mathbb{C}\right)$
is formed by matrices as
\begin{equation}\label{Cartan Dn}
\mathfrak{h}=\{  \mathrm{diag}\left(
\varepsilon_{1},\ldots,\varepsilon_{n},-\varepsilon_{1},\ldots,-\varepsilon_{n}\right) ;\,\varepsilon_{i}%
\in\mathbb{C}\}  .
\end{equation}
The root system of the pair  $\left( \mathfrak{so}\left( 2n,\mathbb{C}\right),\mathfrak{h}\right)$ is given by
\begin{equation}\label{root system Dn}
\Pi=\{  \pm\left(  \varepsilon_{i}\pm\varepsilon_{j}\right);\,1\leq i<j\leq n\}.
\end{equation}
The system of roots for the subalgebra
$\mathfrak{k}_{\Theta}^{\mathbb{C}}=\mathfrak{sl}\left(
n_{1},\mathbb{C}\right) \times\cdots\times\mathfrak{sl}\left(
n_{s},\mathbb{C}\right)$ is
$$
\Pi_{\Theta}=\{  \pm\left(
\varepsilon_{c}^{i}-\varepsilon_{d}^{i}\right) ;\,1\leq c<d\leq
n_{i}\}  ,
$$
then
\[
\Pi_{M}^{+}=\{
\varepsilon_{a}^{i}\pm\varepsilon_{b}^{j};\,1\leq i<j\leq s\}
\cup\{  \varepsilon_{a}^{i}+\varepsilon_{b}^{i};\,a<b\} .
\]

The subalgebra $\mf{t}$ is formed by matrices of the form
\[
\mathfrak{t}=\{\mathrm{diag}\left(
\varepsilon_{n_{1}}^{1},\ldots,\varepsilon_{n_{1}}^{1},\ldots,\varepsilon_{n_{s}}^{s},\ldots,\varepsilon_{n_{s}}^{s},-\varepsilon_{n_{1}}^{1},\ldots
,-\varepsilon_{n_{1}}^{1},\ldots,-\varepsilon_{n_{s}}^{s},\ldots,-\varepsilon_{n_{s}}^{s}\right)\in i\mathfrak{h}_{\mathbb{R}}\}
\]
where $\varepsilon_{n_{i}}^{i}$ appears exactly  $n_{i}$ times. By
restricting the roots of $\Pi_{M}^{+}$ to $\mathfrak{t}$, and
using the notation $\delta_{i}=k(\varepsilon_{a}^{i})$, we obtain
the t-root set:
\[
\Pi^{+}_{\mathfrak{t}}=\{\delta_{i}\pm\delta_{j},2\delta_{i};1\leq
i<j\leq s\}
\]
In particular there exist $s^{2}$ positive t-roots.
\end{proof}

The Killing form on  $\mathfrak{so}\left(
2n\right) $ is given by  $B\left(  X,Y\right)  =2\left(
n-1\right)  trXY$, $B\left(  \alpha ,\alpha\right)
=\frac{1}{2\left(  n-1\right)},$ for all
$\alpha\in\Pi$. The eigenvectors $X_{\alpha}$ satisfying (\ref{base weyl}) are given by
$$
E_{ab}^{ij}=\frac{1}{2\sqrt{n-1}}E_{\varepsilon_{a}^{i}-\varepsilon_{b}^{j}
},\quad F_{ab}^{ij}=\frac{1}{2\sqrt{n-1}}E_{\varepsilon_{a}^{i}
+\varepsilon_{b}^{j}},
$$
$$
 F_{-ab}^{ij}=\frac{1}{2\sqrt{n-1}}E_{-(\varepsilon_{a}^{i}
+\varepsilon_{b}^{j})},\quad G_{ab}^{i}=\frac{1}{2\sqrt{n-1}}E_{\varepsilon
_{a}^{i}+\varepsilon_{b}^{i}}
$$
where $E_\alpha$ denotes the canonical eigenvector of $\mf{g}_{\alpha}$.
The absolute value of structures constants is equal to $
1/2\sqrt{n-1}.$

The notation for the invariant scalar product on the base $\{ X_{\alpha},\alpha\in\Pi_{M}\}$ is given by
\begin{equation}\label{notation Dn}
g_{ij}=g(  E_{ab}^{ij},E_{ba}^{ji})  ,\quad f_{ij}=g(  F_{ab}^{ij}, F_{-ab}^{ij}) ,\quad h_{i}=g( G_{ab}^{i},G_{ba}^{i}),\quad 1\leq i<j\leq s.
\end{equation}
\begin{proposition}
$(\cite{Arv art})$ The Einstein equations on the spaces
$\mathbb{F}_{D}(n_{1},\cdots, n_{s})$ reduce to an algebraic system of $s^{2}$ equations and $s^{2}$ unknowns $g_{ij},$ $f_{ij},h_{i}:$%

\[
n_{i}+n_{j}+\frac{1}{2}\{  \sum_{l\neq i,j}\frac{n_{l}}{g_{il}g_{jl}%
}\left(  g_{ij}^{2}-\left(  g_{il}-g_{jl}\right)  ^{2}\right)
+\sum_{l\neq i,j}\frac{n_{l}}{f_{il}f_{jl}}\left(
g_{ij}^{2}-\left(  f_{il}-f_{jl}\right) ^{2}\right)
\]%
\[
 +\frac{n_{i}-1}{f_{ij}h_{i}}\left(  g_{ij}^{2}-\left(  f_{ij}%
-h_{i}\right)  ^{2}\right)  +\frac{n_{j}-1}{f_{ij}h_{j}}\left(  g_{ij}%
^{2}-\left(  f_{ij}-h_{j}\right)  ^{2}\right)  \}  =g_{ij},
\]

\[
n_{i}+n_{j}+\frac{1}{2}\{  \sum_{l\neq i,j}\frac{n_{l}}{g_{il}f_{jl}%
}\left(  f_{ij}^{2}-\left(  g_{il}-f_{jl}\right)  ^{2}\right)
+\sum_{l\neq i,j}\frac{n_{l}}{f_{il}g_{jl}}\left(
f_{ij}^{2}-\left(  f_{il}-g_{jl}\right) ^{2}\right)
\]%
\[
 +\frac{n_{i}-1}{g_{ij}h_{i}}\left(  f_{ij}^{2}-\left(  g_{ij}%
-h_{i}\right)  ^{2}\right)  +\frac{n_{j}-1}{g_{ij}h_{j}}\left(  f_{ij}%
^{2}-\left(  g_{ij}-h_{j}\right)  ^{2}\right)  \}  =f_{ij},
\]

\[
2\left(  n_{i}-1\right)  +\sum_{l\neq
i}\frac{n_{l}}{g_{il}f_{il}}\left( h_{i}^{2}-\left(
g_{il}-f_{il}\right)  ^{2}\right)  =h_{i}.
\]

\end{proposition}

Now we will treat the case $\mathbb{F}_{D}[n_{1},\ldots, n_{s+1}]$. We consider the same Cartan subalgebra and root system associated given in (\ref{Cartan Dn}) and (\ref{root system Dn}). In this case, the root system for the subalgebra $\mathfrak{k}_{\Theta}^{\mathbb{C}}=\mathfrak{sl}\left(
n_{1},\mathbb{C}\right) \times\cdots\times\mathfrak{sl}\left(
n_{s},\mathbb{C}\right)\times \mathfrak{so}\left( 2n_{s+1},\mathbb{C}\right)$, with respect to $\mathfrak{h}$ is

\[
\Pi_{\Theta}=\{  \pm\left(
\varepsilon_{c}^{i}-\varepsilon_{d}^{i}\right) ;\, 1\leq c<d\leq
n_{i}, 1\leq i\leq s+1\}\cup \{\pm (\varepsilon_c^{s+1}+\varepsilon_d^{s+1});\,1\leq c< d\leq n_{s+1}\}
\]

then

\[
\Pi_{M}^{+}=\{
\varepsilon_{a}^{i}\pm\varepsilon_{b}^{j}:1\leq i<j\leq s+1\}
\cup\{  \varepsilon_{a}^{i}+\varepsilon_{b}^{i}:1\leq a<b\leq n_i, 1\leq i\leq s\} .
\]

The subalgebra $\mf{t}$ is formed by matrices of the form
\[
\mathfrak{t}=\{\mathrm{diag}\left(
\varepsilon_{n_{1}}^{1},\ldots,\varepsilon_{n_{1}}^{1},\ldots,\varepsilon_{n_{s}}^{s},\ldots,\varepsilon_{n_{s}}^{s},-\varepsilon_{n_{1}}^{1},\ldots
,-\varepsilon_{n_{1}}^{1},\ldots,-\varepsilon_{n_{s}}^{s},\ldots,-\varepsilon_{n_{s}}^{s}, 0,\ldots, 0\right)\in i\mathfrak{h}_{\mathbb{R}}\}
\]
where $\varepsilon_{n_{i}}^{i}$ appears exactly  $n_{i}$ times and $0$ appears $t$ times. By
restricting the roots of $\Pi_{M}^{+}$ to $\mathfrak{t}$, and
using the notation $\delta_{i}=k(\varepsilon_{a}^{i})$, we obtain
the t-root set:
\[
\Pi^{+}_{\mathfrak{t}}=\{\delta_{i}\pm\delta_{j},2\delta_{i},\delta_{i}:1\leq
i<j\leq s\}
\]
Note that $k(\varepsilon_{a}^{i}-\varepsilon_{b}^{s+1})=\delta_i$, $1\leq i\leq s$. Then from remark \ref{remark} it follows that
$$
t_i:=g_{i(s+1)}=f_{i(s+1)}\quad \text{and}\quad Ric(E_{ab}^{i(s+1)},E_{ba}^{(s+1)i})=Ric(F_{ab}^{i(s+1)},F_{-ab}^{i(s+1)}),\quad 1\leq i\leq s
$$
here we are keeping the notation (\ref{notation Dn}) for $1\leq i<j\leq s+1$.

In the next table we compute $\Pi_{\Theta}(\alpha)$ and $\Pi_M(\alpha)$ for each $\alpha\in\Pi_M^+$.

\begin{landscape}
\begin{small}
\begin{tabular}{ccccccc}
  \hline
  $\alpha\in \Pi_{M}^+$ &  &  &$\Pi_{\Theta}\left(\alpha\right)$ is the union of &  & &$\Pi_{M}\left(\alpha\right)$ is the union of  \\
  \hline
$\varepsilon_a^i-\varepsilon_b^j$  &  &  &$\{\varepsilon_c^i-\varepsilon_a^i:\scriptsize{1\leq c\leq n_i, c\neq a}\}$  &  &  &$\{\varepsilon_b^j\pm\varepsilon_c^k,\varepsilon_c^k-\varepsilon_a^i\}$
\\
\footnotesize{$1\leq i<j\leq s$}    &  &  &$\{\varepsilon_b^j-\varepsilon_c^j:\scriptsize{1\leq c\leq n_j, c\neq b}\}$   &  &  & \footnotesize{$1\leq k\leq s+1, k\neq i,j$}  \\
&  &  &  &  &  &$\{\varepsilon_b^j+\varepsilon_c^j:\scriptsize{1\leq c\leq n_j, c\neq b}\}$  \\
&  &  &  &  &  & $\{-(\varepsilon_a^i+\varepsilon_c^i):\scriptsize{1\leq c\leq n_i, c\neq a}\}$  \\
\hline
$\varepsilon_a^i-\varepsilon_b^{s+1}$  &  &  &$\{\varepsilon_b^{s+1}\pm\varepsilon_c^{s+1}:\scriptsize{1\leq c\leq n_{s+1}, c\neq b}\}$  &  &  & $\{\varepsilon_b^{s+1}\pm\varepsilon_c^k,\varepsilon_c^k-\varepsilon_a^i,-(\varepsilon_c^k+\varepsilon_a^i)\}$
\\
\footnotesize{$1\leq i\leq s$} &  &  &$\{\varepsilon_d^i-\varepsilon_a^i:\scriptsize{1\leq d\leq n_i, d\neq a}\}$  &  &  &\footnotesize{$1\leq k\leq s, k\neq i$}   \\
 &  &  &   &  &  & $\{-(\varepsilon_a^i+\varepsilon_c^i),\varepsilon_b^{s+1}+\varepsilon_c^i:\scriptsize{1\leq c\leq n_i, c\neq a}\}$    \\
 \hline
$\varepsilon_a^i+\varepsilon_b^j$  &  &  &$\{\varepsilon_c^i-\varepsilon_a^i:\scriptsize{1\leq c\leq n_i, c\neq a}\}$  &  &  & $\{\varepsilon_c^k-\varepsilon_a^i,-(\varepsilon_c^k+\varepsilon_a^i),\varepsilon_c^k-\varepsilon_b^j, -(\varepsilon_c^k+\varepsilon_b^j)\}$  \\
\footnotesize{$1\leq i<j\leq s$} &  &  &$\{\varepsilon_c^j-\varepsilon_b^j:\scriptsize{1\leq c\leq n_j, c\neq b}\}$   &  &  & \footnotesize{$1\leq k\leq s+1; k\neq i,j$}  \\
 &  &  &  &  &  & $\{-(\varepsilon_a^i+\varepsilon_c^i),\varepsilon_c^i-\varepsilon_b^j:1\leq c\leq n_i, c\neq a\}$  \\
 &  &  &  &  &  & $\{-(\varepsilon_b^j+\varepsilon_d^j),\varepsilon_d^j-\varepsilon_a^i:1\leq d\leq n_j, d\neq b\}$  \\
  \hline
  $\varepsilon_a^i+\varepsilon_b^{s+1}$  &  &  &$\{\varepsilon_c^i-\varepsilon_a^i:1\leq c\leq n_i, c\neq a\}$  &  &  &$\{\varepsilon_c^k-\varepsilon_a^i, -(\varepsilon_c^k+\varepsilon_a^i),\varepsilon_c^k-\varepsilon_b^{s+1},-(\varepsilon_c^k+\varepsilon_b^{s+1})\}$   \\
\footnotesize{$1\leq i\leq s$} &  &  &$\{-(\varepsilon_b^{s+1}+\varepsilon_d^{s+1}),\varepsilon_d^{s+1}-\varepsilon_b^{s+1}\}$  &  &  & \footnotesize{$1\leq k\leq s, k\neq i$}  \\
 &  &  &\footnotesize{$1\leq d\leq n_{s+1}, d\neq b$}  &  &  & $\{-(\varepsilon_a^i+\varepsilon_d^i),\varepsilon_d^i-\varepsilon_b^{s+1}:1\leq d \leq n_i,d\neq a \}$  \\
 \hline
  $\varepsilon_a^i+\varepsilon_b^i$  &  &  &$\{\varepsilon_c^i-\varepsilon_a^i,\varepsilon_c^i-\varepsilon_b^i\}$  &  &  & $\{\varepsilon_c^k-\varepsilon_a^i,\varepsilon_c^k-\varepsilon_b^i, -(\varepsilon_c^k+\varepsilon_a^i),-(\varepsilon_c^k+\varepsilon_b^i)\}$   \\
  \footnotesize{$1\leq a<b\leq n_i;\scriptsize{1\leq i\leq s}$} &  &  &  \footnotesize{$1\leq c\leq n_i;c\neq a,b$}&  &  &\footnotesize{$1\leq k\leq s+1;k\neq i$}   \\
  &  &  &  &  &  &   \\
   \hline
\end{tabular}
\end{small}
\end{landscape}

\begin{proposition} \label{Einstein equation Dn 2}
The Einstein equation on $\mathbb{F}_{D}[n_{1},\ldots, n_{s+1}]$ reduces to the following algebraic system with $s^2+s$ unknowns $g_{ij}$, $f_ij$, $h_i$, $t_i$ and $s^2+s$ equations.

\begin{align*}
& n_i+n_j +\frac{1}{4}\left\{\sum_{k\neq i,j}^{s+1}\frac{n_k}{f_{ik}f_{jk}}\left(g_{ij}^2-(f_{ik}-f_{jk})^2\right)+2\sum_{k\neq i,j}^{s+1}\frac{n_k}{g_{ik}g_{jk}}\left(g_{ij}^2-(g_{ik}-g_{jk})^2\right)\right.\\
\\
&\left.
+ \frac{n_j-1}{f_{ij}h_{j}}\left(g_{ij}^2-(f_{ij}-h_{j})^2\right)
+\frac{n_i-1}{f_{ij}h_{i}}\left(g_{ij}^2-(f_{ij}-h_{i})^2\right)\right\}=g_{ij}
\end{align*}
\begin{align*}
&2n_{s+1}+ n_i-1+\frac{1}{2}\left\{\sum_{k\neq i}^{s}\frac{n_k\left(t_i^2-(f_{ik}-t_k)^2\right)}{f_{ik}t_k}+\sum_{k\neq i}^{s}\frac{n_k\left(t_i^2-(g_{ik}-t_k)^2\right)}{g_{ik}t_k}\right. \\
\\
&\left. +\frac{n_i-1}{t_i h_i}\left(t_i^2-(t_i-h_i)^2\right)\right\}=t_i
\end{align*}
\begin{align*}
& n_i+n_j +\frac{1}{2}\left\{\sum_{k\neq i,j}^{s+1}\frac{n_k}{f_{jk}g_{ik}}\left(f_{ij}^2-(f_{jk}-g_{ik})^2\right)+\sum_{k\neq i,j}^{s+1}\frac{n_k}{f_{ik}g_{jk}}\left(f_{ij}^2-(f_{ik}-g_{jk})^2\right)\right.\\
\\
&\left.+\frac{n_i-1}{g_{ij}h_{i}}\left(f_{ij}^2-(g_{ij}-h_{i})^2\right)+\frac{n_j-1}{g_{ij}h_j}\left(f_{ij}^2-(g_{ij}-h_j)^2\right)\right\}=f_{ij}
\end{align*}
\begin{align*}
& n_i+\sum_{k\neq i}^{s+1}\frac{n_k}{f_{ik}g_{ik}}\left(h_i^2-(f_{ik}-g_{ik})^2\right)=h_i
\end{align*}
\end{proposition}

\section{Number of isotropy summands}

The description of several invariant tensor on a flag manifold depend on the number of isotropy summands. We computed this number above for some families of flag manifolds. The next result provides the amount of isotropy summands (or positive t-roots) for any classical flag manifold.
\begin{theorem}\label{number of isotropy summands}
The number of isotropy summands of flag manifolds of classical Lie groups is determined by the following table.
\newline\newline%
\begin{tabular}
[c]{cccccc}%
\multicolumn{6}{l}{Table 1}
\\\hline Type &  & $G/K$ &
&  & $\left\vert \Pi_{\mathfrak{t}}^{+}\right\vert $\\\hline
&  &  &  &  & \\
$A_{n}$ &  & $\mathbb{F}_{A}(n_{1},\ldots,n_{s})$ &  &  &
$s\left(  s-1\right)/2$\\
&  &  &  &  & \\\hline
&  &  &  &  & \\
& $\left(  i\right)  $ & $SO\left(  2n+1\right)/U\left(
n_{1}\right)  \times\cdots\times U\left(  n_{s}\right)  \times
U\left(1\right)  ^{m}$ &  &  & $\left(  s+m\right)  ^{2}+s$\\
&  & \footnotesize{$n\geq2$} &  &  & \\
$\text{ }B_{n}$ &  &  &  &  & \\
& $\left(  ii\right)  $ & $SO\left(  2n+1\right)/U\left(
n_{1}\right)  \times\cdots\times U\left(  n_{s}\right)  \times
U\left( 1\right)  ^{m}\times SO\left(  2t+1\right)  $ &  &  &
$\left(  s+m\right)
^{2}+s$\\
&  & \footnotesize{$t\geq2$}  &  &  & \\
&  &  &  &  & \\\hline
&  &  &  &  & \\
& $\left(  i\right)  $ & $Sp(n)/U(n_{1})\times\cdots\times
U(n_{s})\times U(1)^{m}$ &  &  & $\left(  s+m\right)  ^{2}$\\
&  & \footnotesize{$n\geq3$} &  &  & \\
$C_{n}$ &  &  &  &  & \\
& $\left(  ii\right)  $ & $Sp(n)/U(n_{1})\times\cdots\times
U(n_{s})\times U(1)^{m}\times Sp(t)$ &  &  & $\left(  s+m\right)
^{2}+\left(  s+m\right)  $\\
&  & \footnotesize{$t\geq3$} &  &  & \\
&  &  &  &  & \\\hline
&  &  &  &  & \\
& $\left(  i\right)  $ & $SO(2n)/U(n_{1})\times\cdots\times
U(n_{s})\times U(1)^{m}$ &  &  & $\left(  s+m\right)  ^{2}-m$\\
$D_{n}$ &  & \footnotesize{$n\geq4$} &  &  & \\
&  &  &  &  & \\
& $\left(  ii\right)  $ &
$SO(2n)/U(n_{1})\times\cdots\times U(n_{s})\times
U(1)^{m}\times SO(2t)$ &  &  & $\left(  s+m\right)
^{2}+s$\\
&  & \footnotesize{$t\geq4$} &  &  & \\
&  &  &  &  & \\\hline
\end{tabular}
\newline\newline
where ${\textstyle\sum}n_{i}+m=n$, $n_{i}>1$ and $m,s\geq0$ in the cases
$B_{n}\left( i\right)$, $C_{n}\left(  i\right)$ and $D_{n}\left(
i\right)$. For $B_{n}\left(  ii\right)$, $C_{n}\left(
ii\right)$ and $D_{n}\left( ii\right)$ we have $\sum n_{i}+t+m=n$, with
$n_{i}>1$ and $m,s\geq 0$.
\end{theorem}

\begin{proof}
The calculus of  $\left\vert\Pi_{\mathfrak{t}}^{+}\right\vert$ is analogous for all cases. Thus we present the computation only for $B_{n}\left( ii\right)$. In this case, we consider the Cartan subalgebra $\mathfrak{h}$ given in (\ref{Cartan subalgebra Bn}) with root system (\ref{Pi Bn}).
The root system of
$\mathfrak{k}_{\Theta}^{\mathbb{C}}=\mathfrak{gl}\left(
n_{1},\mathbb{C}\right)  \times\cdots\times\mathfrak{gl}\left(  n_{s}%
,\mathbb{C}\right)  \times\mathfrak{gl}\left(  1,\mathbb{C}\right)  ^{m}\times\mathfrak{so}\left(  2t+1,\mathbb{C}\right)$ with respect to  $\mathfrak{h}$ is
\begin{align*}
\Pi_{\Theta}&=\{  \pm\left(
\varepsilon _{a}^{i}-\varepsilon_{b}^{i}\right)  :1\leq i\leq
s,1\leq a<b\leq
n_{i}\}\\
&\cup\{  \pm\left(  \varepsilon_{a}^{s+m+1}%
\pm\varepsilon_{b}^{s+m+1}\right)
,\pm\varepsilon_{c}^{s+m+1}:1\leq a<b\leq t,1\leq c\leq t\},
\end{align*}
where $\varepsilon_{c}^{s+m+1}=\varepsilon
_{n_{1}+\cdots+n_{s}+m+c},1\leq c\leq t$. Then
\[
\mathfrak{t}=\{
\mathrm{diag}(\varepsilon^{1}_{n_{1}},\ldots,\varepsilon^{1}_{n_{1}},\ldots,\varepsilon^{s}_{n_{s}}
,\ldots,\varepsilon^{s}_{n_{s}},\varepsilon_{n-t-m+1},\varepsilon_{n-t-m+2},\ldots,\varepsilon_{n-t}, 0,\ldots,0)\in i\mathfrak{h}_{\mathbb{R}}\}
\]
where each $\varepsilon^{i}_{n_{i}}$ and $0$ appears exactly
$n_{i}$ times and $t$ times, respectively.

Considering the choice of positive (\ref{positive root Bn}) we see that the positive complementary roots is
\begin{align*}
\Pi_{M}^{+}&=\{ \varepsilon_{a}^{i}\pm\varepsilon
_{b}^{j}:1\leq i<j\leq s\}  \cup\{  \varepsilon_{c}^{i}
,\varepsilon_{a}^{i}+\varepsilon_{b}^{i},\varepsilon_{d}^{s+1}:a<b, \quad 1\leq
d\leq
m\}\\
&\cup\{  \varepsilon_{a}^{i}\pm\varepsilon_{d}^{s+1}
,\varepsilon_{a}^{i}\pm\varepsilon_{c}^{s+m+1},\varepsilon_{d}^{s+1}
\pm\varepsilon_{c}^{s+m+1}\} \cup\{ \varepsilon
_{c}^{s+1}\pm\varepsilon_{d}^{s+1}:c<d\}.
\end{align*}
So by restricting roots of $\Pi_{M}^{+}$ to the subalgebra
$\mathfrak{t}$, we obtain the positive t-roots set:
\[
\Pi_{\mathfrak{t}}^{+}=\{  \delta_{i}\pm \delta_{j}:1\leq
i<j\leq s\} \cup\{
\delta_{i},2\delta_{i},\delta_{n-t-m+k}:1\leq i\leq s,1\leq k\leq
m\}
\]
\[
\cup\{  \delta_{i}\pm \delta_{n-t-m+k},:1\leq i\leq s,1\leq
k\leq m\}\cup\{  \delta_{n-t-m+k}\pm
\delta_{n-t-m+r}:1\leq k<r\leq m\} .
\]
where $\delta_{i}$ denotes the restriction
$k(\varepsilon_{a}^{i})$. Note that $k(
\varepsilon_{a}^{i}\pm\varepsilon_{c}^{s+m+1})
=\delta_{i}=k(\varepsilon_{a}^{i})$
and $k(
\varepsilon_{d}^{s+1}\pm\varepsilon_{c}^{s+m+1})
=\delta_{n-t-m+k}=k(\varepsilon_{d}^{s+1})$. So  $\left\vert \Pi_{\mathfrak{t}}^{+}\right\vert =\left( s+m\right)
^{2}+s$.

Now, for the other cases, using the notation $\delta_{i}:=k(\varepsilon^{i}_{a})$, we
can write
\begin{align*}
\Pi_{\mathfrak{t}}^{+}  &  =\left\{
\begin{array}
[c]{c}%
\sigma\left(  \delta_{i}-\delta_{j}\right)  ,\alpha\left(
\delta_{i}+\delta_{j}\right)
,\eta\left(  \delta_{n-t-m+k}\pm \delta_{n-t-m+r}\right)  :\\
1\leq i<j\leq s,1\leq r<k\leq m
\end{array}
\right\} \\
&  \bigcup\left\{
\begin{array}
[c]{c}%
\beta \delta_{i},\gamma2\delta_{i},\mu
\delta_{n-t-m+k},\xi2\delta_{n-t-m+k},\psi\left(
\delta_{i}+\delta_{n-t-m+k}\right)  ,\\
\zeta\left(  \delta_{i}-\delta_{n-t-m+k}\right)  :1\leq i\leq
s,1\leq k\leq m
\end{array}
\right\}
\end{align*}
where $\sigma,\alpha,\beta,\gamma,\mu,\xi,\psi,\zeta,\eta$ are 0
or 1.

In the case $A_{n}$, only  $\sigma= 1$ and the remaining
coefficients are zero.

By analyzing the case $B_{n}\left( i\right)$,
when $s=0$, (i.e. full flag), we see that only  $\mu=\eta= 1$. Besides, if  $m=0$ then
$\sigma=\alpha=\beta=\gamma= 1$. And when $s,m\neq0$ then only
$\xi= 0$.

In the case $B_{n}\left( ii\right)$, if $s=0$ then only
$\eta=\mu=1$. If $m=0$ then $\sigma=\alpha=\beta=\gamma=1$.

For flags manifold of type $C_{n}\left( i\right) $, if $s=0$ then
$\xi,\eta=1$. If $m=0$ then $\sigma=\alpha=\gamma=1$. When
$s,m\neq0$ then $\beta=\mu=0$. Now for space of $C_{n}\left( ii\right)$, if $s=0$ then
$\eta=\mu=\xi=1$. If $m=0$ then $\sigma=\alpha=\beta=\gamma=1$.
When $s,m\neq0$ then all coefficients are equal 1.

By studying the case $D_{n}\left( i\right)$, if $s=0$ then only $\eta= 1$. When $m=0$ we have $\sigma=\alpha=\gamma=1$. If $s,m\neq0$ then
$\beta=\mu=\xi=0$. Finally, for flags manifold of type $D_{n}\left( ii\right)$, if
$s=0$ then $\eta=\mu= 1$. When $m=0$ then
$\sigma=\alpha=\beta=\gamma=1$. If $s,m\neq0$ then only $\xi=0$.
This concludes the proof.
\end{proof}


\end{document}